\documentclass[11pt,a4paper]{amsart}

\usepackage[text={140mm,205mm},centering]{geometry}

\usepackage{amsfonts,amsthm,amsmath,amssymb}
\usepackage{amscd}
\usepackage[comma,sort,compress]{natbib}

\usepackage{hyperref} 
\hypersetup{ colorlinks=true, linkcolor=blue, citecolor=blue,
  filecolor=blue, urlcolor=blue}

\newcommand{\nc}{\newcommand}
\newcommand{\sgn}{\operatorname{sgn}}

\nc{\io}{\iota}

\newtheorem{theorem}{Theorem}[section]
\newtheorem{lemma}{Lemma}[section]
\newtheorem{prop}{Proposition}[section]

\newtheorem{fact}{Fact}[section]

\newtheorem{defn}{Definition}

\newtheorem{remark}{Remark}[section]
\newtheorem{corollary}{Corollary}
\newtheorem{example}{Example}

\newcommand{\weak}{\stackrel{w}{\longrightarrow}}
\newcommand{\prob}{\stackrel{P}{\longrightarrow}}
\newcommand{\eid}{\stackrel{d}{=}}
\newcommand{\one}{{\bf 1}}

\newcommand{\reals}{{\mathbb R}}
\newcommand{\bbr}{\reals}
\newcommand{\vep}{\varepsilon}
\newcommand{\bbz}{\protect{\mathbb Z}}
\newcommand{\bbn}{\protect{\mathbb N}}
\newcommand{\bbc}{\protect{\mathbb C}}

\newcommand{\ol}{\overline}

\def\Var{{\rm Var}}
\def\E{{\rm E}}
\def\esd{{\rm ESD}}
\def\eim{{\rm EM}}

\DeclareMathOperator{\Tr}{Tr} 
\DeclareMathOperator{\rank}{Rank}

\numberwithin{equation}{section}
\newcommand{\newcomment}[1]{}
\newcommand{\comment}[1]{}

\begin{document} 

\bibliographystyle{abbrvnat}

\title[Random matrices to long range dependence]{From random matrices to long range dependence}

\author[A. Chakrabarty]{Arijit Chakrabarty}
\address{Theoretical Statistics and Mathematics Unit, Indian Statistical Institute, New Delhi}
\email{arijit@isid.ac.in}

\author[R. S. Hazra]{Rajat Subhra Hazra}
\address{Theoretical Statistics and Mathematics Unit, Indian Statistical Institute, Kolkata}
\email{rajatmaths@gmail.com}

\author[D. Sarkar]{Deepayan Sarkar}
\address{Theoretical Statistics and Mathematics Unit, Indian Statistical Institute, New Delhi}
\email{deepayan@isid.ac.in}



\begin{abstract}
  Random matrices whose entries come from a stationary Gaussian process are
studied. The limiting behavior of the eigenvalues as the size of the
matrix goes to infinity is the main subject of interest in this work. It
is shown that the limiting spectral distribution is determined by the absolutely continuous component  of
the spectral measure of the stationary process. This is similar to the situation where the entries of the matrix are i.i.d.
On the other hand, the discrete component contributes to the  limiting
behavior of the eigenvalues after a different scaling. Therefore, this helps to
define a boundary between short and long range dependence of a stationary
Gaussian process in the context of random matrices.
\end{abstract}

\maketitle

\section{Introduction}
\label{sec:intro}

The notion of {\it long range dependence} is of significant importance
in the field of stochastic processes. Consider any stationary
stochastic process indexed by $\bbz$. If the process is an
i.i.d.\ collection, then it does not have any memory, and hence it is
{\it short range dependent}. A stark contrast to the i.i.d. situation is one where any two entries of the stochastic process are almost surely equal, in which case it is long range dependent.  For a general stochastic process which is not necessarily one of the these two extremes, whether it is long or short range dependent is
determined by whether it resembles an i.i.d. collection or a collection where the same random variable gets repeated. In order to
make the idea of resemblance precise, different functionals of the
process are studied. If the behavior of a functional of interest is
close to that in the i.i.d.\ setup, then the process is short range
dependent, otherwise it is long range dependent. Naturally, the
definition of long range dependence varies widely with context, and it
is not surprising that there are numerous definitions of this concept in
the literature, which are not equivalent. The survey article by
\cite{samorodnitsky:2006} describes in detail this notion from various
points of view.

The current paper is an attempt to understand long range dependence in
yet another context, namely that of random matrices. Let
$\{X_{j,k}:j,k\in\bbz\}$ be a {\bf real} stationary Gaussian process with zero
mean and positive variance. That means,
\[
 \E(X_{j,k})=0\,,
\]
\[
 \E\left(X_{j,k}^2\right)>0\,,
\]
and
\[
 \E\left(X_{j,k}  {X_{j+u,k+v}}\right)
\]
is independent of $j$ and $k$ for all
 fixed $u,v\in\bbz$. For $N\ge1$, define an $N\times N$ matrix $W_N$ by
\begin{equation}\label{eq.defan}
 W_N(i,j):=   X_{i,j}+{X_{j,i}}\,,
\end{equation}
for all $1\le i,j\le N$. Clearly, $W_N$ is symmetric by construction,
and hence its eigenvalues are all real. For any $N\times N$ symmetric
matrix $A$, denote its eigenvalues by
$\lambda_1(A)\le\ldots\le\lambda_N(A)$, and define its empirical spectral
distribution, henceforth abbreviated to $\esd$, by
\[
 \esd(A):=\frac1N\sum_{j=1}^N\delta_{\lambda_j(A)}\,.
\]

Section \ref{sec:result} lists the main results of the paper. That section is divided into three subsections.
In Subsection \ref{subsec:esd}, the results that study the limit of $\esd(W_N/\sqrt N)$ as
$N\to\infty$ are listed, the main result being Theorem \ref{t1}. In
Subsection \ref{subsec:em}, a variant of the $\esd$ called ``eigen measure'' is
defined. The main result of that subsection, Theorem \ref{discrete.t1},
studies the limit of the eigen measure of $W_N/N$ as $N\to\infty$. The
above two theorems motivate a natural definition of long range
dependence, which is discussed in Subsection \ref{subsec:lrd}. Section \ref{example} contains a corollary
and a few examples. The proofs of the results mentioned in Section \ref{sec:result} are given in Section \ref{sec:proof}.

We would like to point out that Theorem \ref{t1} is actually
an extension of the classical result by Wigner which says that if
$X_{i,j}$ are i.i.d.\ standard normal random variables, then
$\esd(W_N/\sqrt N)$ converges to the Wigner semicircle law (defined in \eqref{eq.defwsl}). 
Relaxation of the independence assumption has previously been investigated by
\cite{banna}, \cite{BKV},  \cite{BKA}, \cite{KM2008}, \cite{cha}, \cite{ GT},   \cite{magda}, \cite{marlevede:peligrad:peligrad:2015}, \cite{gotze:naumov:tikhomirov:2015} and \cite{RTWR}. 
\comment{correction-40}
The articles by \cite{Ngyuen:Rourke, adam, HLN, Naumov} and \cite{OE2012}
 have studied the sample covariance
matrix  and non-symmetric matrices after imposing some dependence structures.  A  work by \cite{anderson:zeitouni:2008}, which is  related to the current paper, considered the $\esd$ of Wigner matrices where on and off diagonal elements form a finite-range dependent random field; in
particular, the entries are assumed to be independent beyond a finite range,
and within the finite range the correlation structure is given by a kernel
function.  It should be noted that the results therein also apply for entries that are not necessarily Gaussian. The results of the current paper are more general but apply only to Gaussian matrices.
\newcomment{Comment 1}

 In particular, suppose that we are interested in finding the conditions on the
 process $(X_{i,j})$ under which one can compute the limiting $\esd$.
 Previous attempts have led to conditions that are equivalent to (or
 stronger than) assuming that the correlations are summable.  Our paper
 breaks this barrier by observing that one can go further, that is,
 just require that the spectral measure of the process does not contain
 a continuous singular component.  Needless to say, the summability of
 correlations does imply absolute continuity of the spectral measure,
 but the converse is not true. Therefore, the results in this paper are generalizations of some of those in the literature. In all the examples mentioned in Section \ref{example}, the correlations are not summable, and hence the results proved here are necessary for dealing with them.
 
 We also study the effect of the presence of a discrete component of
 the spectral measure, which does not affect the ESD but gives rise to
 interesting limiting behavior for the eigen measure.  If the random variables were i.i.d., then the random matrix under study would have a non-trivial limiting ESD, but the limiting eigen measure would have been trivial. On the other hand, if the correlation between any two entries of the matrix were one, that is, the entries were all same, then the limiting eigen measure would have been non-trivial, but not the limiting ESD. This illustrates that there is a natural divide between the processes whose spectral measure is absolutely continuous, and the ones whose spectral measure is discrete, the former resembling the i.i.d. case while the latter is similar to the situation that all the correlations are one.This is precisely the connection between the random matrix model and  long range dependence, explained further in Subsection \ref{subsec:lrd}. 

It turns out that in some situations, the limiting spectral distribution is the free multiplicative convolution of the Wigner semicircle law and a probability measure on the positive half line. Thus, the results of this paper can be used to answer questions from free probability, as has been done in \cite{chakrabarty:hazra:2016}. However, these constitute an area in their own right worth studying, and therefore, has not been touched in this paper.

\section{The results}\label{sec:result}

Define
\[
 R(u,v):=\E\left(X_{0,0}  {X_{u,v}}\right),\,u,v\in\bbz\,.
\]

The Herglotz representation theorem asserts that there exists a finite measure $\nu$ on $(-\pi,\pi]^2$ such that
\begin{equation}\label{p1.eq1}
 R(u,v)=\int_{(-\pi,\pi]^2} e^{\io(ux+vy)}\nu(dx,dy)\text{ for all }u,v\in\bbz\,,
\end{equation}
where $\io:=\sqrt{-1}$. Let  $\nu_{ac}$, $\nu_{cs}$ and $\nu_d$ denote the components of $\nu$ which are absolutely continuous with respect to the Lebesgue measure, continuous and singular with respect to the Lebesgue measure, that is, supported on a set of Lebesgue measure zero, and discrete, that is, supported on a countable set, respectively. Since $\nu_{ac}$ is absolutely continuous with respect to the Lebesgue measure, there exists a function $f$ from $[-\pi,\pi]^2$ to $[0,\infty)$ such that
\begin{equation}\label{eq.deff}
 \nu_{ac}(dx,dy)=f(x,y)dxdy\,.
\end{equation}

The one and only {\bf assumption} of this paper is that the continuous and singular component is absent, that is,
\begin{equation}\label{eq.assume}
 \nu_{cs}\equiv0\,.
\end{equation}
As a consequence, it follows that
\begin{equation*}
 \nu=\nu_{ac}+\nu_d\,.
\end{equation*}

\subsection{The empirical spectral distribution}\label{subsec:esd}

Denote 
\begin{equation}\label{eq.defmun}
 \mu_N:=\esd(W_N/\sqrt N),\,N\ge1\,,
\end{equation}
where $W_N$ is as in \eqref{eq.defan}.

The task of this subsection is to list the results that study the limiting spectral distribution (henceforth LSD) of $W_N/\sqrt N$, that is, the limit of the random probability measures $\mu_N$ as $N\to\infty$. The first result, Theorem \ref{t1} below, establishes that the limit exists. 

\begin{theorem}\label{t1}
 There exists a deterministic probability measure $\mu_f$, determined solely by the spectral density $f$ which is as in \eqref{eq.deff}, such that
 \[
  \mu_N\to\mu_f\,,
 \]
weakly in probability as $N\to\infty$. By saying that the LSD $\mu_f$ is determined by $f$, the following is meant. If for two stationary processes satisfying assumption \eqref{eq.assume}, the absolutely continuous component of the corresponding spectral measures match, then the LSD of the scaled symmetric random matrices formed by them also agree.
\end{theorem}

\begin{remark}
 The exact description of $\mu_f$ is complicated, and will come much later in Remark \ref{rem:muf}. However, it is shown in \cite{chakrabarty:hazra:2016} that $\mu_f$ is absolutely continuous with respect to the Lebesgue measure if
\[
 {ess\inf}_{(x,y)\in[-\pi,\pi]^2}\left[f(x,y)+f(y,x)\right]>0\,,
\]
where ``$ess\inf$'' denotes the essential infimum. 
\end{remark}

A natural question at this stage is ``When is the probability measure $\mu_f$ degenerate at zero?'' . \comment{Correction-1} 
The following result answers this question.

\begin{theorem}\label{t6}
 The second moment of the probability measure $\mu_f$ is given by
\[
 \int_\bbr x^2\mu_f(dx)=2\int_{[-\pi,\pi]^2} f(x,y)dxdy\,.
\]
\end{theorem}

The next result relates some more properties of $\mu_f$ with those of $f$. 

\begin{theorem}\label{t2} For $m\ge2$, the $(2m)$-th moment of $\mu_f$ is finite if  $\|f\|_m<\infty$. Here $\|f\|_p$ denotes the $L^p$ norm of $f$ for all $p\in[1,\infty]$.
\end{theorem}

In the situation when $f$ is essentially bounded above, $\mu_f$ can be understood better, as illustrated in the next two results. Some new notations will be needed for stating those results, which we now introduce. 

For $m\ge 1$, denote by $NC(2m)$  the set of non-crossing partitions of $\{1,2,\cdots, 2m\}$ and let it be equipped with the partial order $\preccurlyeq$, where $\pi\preccurlyeq \sigma$ means  every block of $\pi$ is completely contained in a block of $\sigma$, making $NC(2m)$ a lattice. An element $a$ of a subset $A$ of $NC(2m)$ is maximal if
\[
b\preccurlyeq a\text{ for all }b\in A\,.
\]  
Let $\pi$ be a partition in $NC(2m)$.  Then the Kreweras complement $K(\pi)$ is defined as the maximal element $\sigma \in NC(\{\overline 1,\ldots,\overline {2m}\})$ with the property that $\pi\cup\sigma \in NC(\{1,\overline1,\ldots,2m,\overline{2m}\})$.\newcomment{Comment 2}
In other words, for all $\sigma^\prime\in NC(\{\ol1,\ldots,\ol{2m}\})$ such that $\pi\cup\sigma^\prime\in NC(\{1,\ol1,\ldots,2m,\ol{2m}\})$, it is necessary that
\[
\sigma^\prime\preccurlyeq K(\pi)\,.
\]
It can be shown that for any $\pi \in NC(2m)$, $|\pi|+ |K(\pi)|= 2m+1$. 

Fix $m\ge1$ and $\sigma\in NC_2(2m)$, the set of non-crossing pair partitions of\\ \label{pg.combinatorics}
$\{1,\ldots,2m\}$. Let $K(\sigma):=(V_1,\ldots,V_{m+1})$ denote the Kreweras
complement of $\sigma$.
In order to ensure uniqueness in the notation, we impose the requirement that the blocks $V_1,\ldots,V_{m+1}$ are ordered in the following way. If $1\le i<j\le m+1$, then the {\bf maximal} element of $V_i$ is strictly less than that of $V_j$. Let $\mathcal T_\sigma$ be the unique  function from $\{1,\ldots,{2m}\}$ to $\{1,\ldots,m+1\}$ satisfying
\begin{equation}\label{eq.defTsigma}
\ol i\in V_{\mathcal T_\sigma(i)},\,1\le i\le2m\,.
\end{equation}
For example, if
\[
 \sigma:=\{(1,4),(2,3),(5,6)\}\,,
\]
\comment{correction-2(b)}
then the Kreweras complement is given by 
\[
K(\sigma)=\{(\ol1,\ol3), (\ol2), (\ol4,\ol6), (\ol5)\}\,,
\]
and hence the ordered blocks are $V_1=(\ol2)$, $V_2=(\ol1,\ol3)$, $V_3=(\ol5)$ and $V_4=(\ol4,\ol6)$. So we have $\mathcal T_\sigma(1) = 2, \mathcal T_\sigma(2) = 1, \mathcal
T_\sigma(3) = 2, \mathcal T_\sigma(4) = 4, \mathcal T_\sigma(5) = 3,
\mathcal T_\sigma(6) = 4$. For details of properties of Kreweras complement we refer to \cite[Chapter 9]{nica:speicher:2006}. Chapter 22 of the said reference uses non-crossing pair partitions and Kreweras complements in the context of random matrices. \comment{correction-2(c)}  

For  any function $f$ from $[-\pi,\pi]^2$ to $\bbr$ and any $\sigma \in NC_2(2m)$ define the function $L_{\sigma,f}$ from $[-\pi,\pi]^{m+1}$ to $\bbr$ by
\[
 L_{\sigma,f}(x):=\prod_{(u,v)\in\sigma}\left[f\left(x_{\mathcal T_\sigma(u)},-x_{\mathcal T_\sigma(v)}\right)+f\left(-x_{\mathcal T_\sigma(v)},x_{\mathcal T_\sigma(u)}\right)\right]\,,
\]
for all $x\in[-\pi,\pi]^{m+1}$.

\begin{theorem}\label{theorem.new}
 Assume that $\|f\|_\infty<\infty$. Then, the following are true.\\
1. The support of $\mu_f$ is contained in $[-\bar R,\bar R]$ where
\[
 \bar R:=4\sqrt2\,\pi\sqrt{\|f\|_\infty}\,.
\]
2. The $(2m)$-th moment of $\mu_f$ is given by 
\[
 \int_\bbr x^{2m}\mu_f(dx)=(2\pi)^{m-1}\sum_{\sigma\in NC_2(2m)}\int_{[-\pi,\pi]^{m+1}}L_{\sigma,f}(x)dx\text{ for all }m\ge1\,.
\]
\end{theorem}

\begin{remark}
It is shown in Example \ref{example1} that the converses of Theorem \ref{t2} and the first statement in Theorem \ref{theorem.new} are false. That is, finiteness of the $(2m)$-th moment of $\mu_f$ does not imply that $\|f\|_m<\infty$ for any $m\ge2$, and $\mu_f$ being compactly supported does not imply that $f$ is essentially bounded above.
\end{remark}

Let  $\|f\|_\infty<\infty$.
The Stieltjes transform $\mathcal G$ of the LSD $\mu_f$ is defined by
\[
 {\mathcal G}(z):=\int_\bbr\frac1{z-x}\mu_f(dx),\,z\in\bbc,|z|>\bar R\,,
\]
where $\bar R$ is as in Theorem \ref{theorem.new}. In view of the first claim of Theorem \ref{theorem.new}, the above integral makes sense.
The next result describes the Stieltjes transform of $\mu_f$.

\begin{theorem}\label{thm.stieltjes}
Assume that $\|f\|_\infty<\infty$. Define ${\mathcal D}:=\{z\in\bbc:|z|>\bar R\}\times[0,1]$.
Then, there exists a unique function $\mathcal H$ from $\mathcal D$ to $\bbc$ such that
\begin{enumerate}
\item for all fixed $x$, ${\mathcal H}(\cdot,x)$ is analytic on $[|z|>\bar R]$,
\item for all fixed $z$, ${\mathcal H}(z,\cdot)$ is Lebesgue integrable on $[0,1]$,
\item for all $(z,x)\in\mathcal D$,
\begin{equation}\label{thm.stieltjes.claim1}
 z{\mathcal H}(z,x)=1+4\pi^2{\mathcal H}(z,x)\int_{-\pi}^\pi {\mathcal H}(z,y)\left[f(x,-y)+f(-y,x)\right]dy\,,
\end{equation}
\item and, for all $x\in[0,1]$,
\[
 \lim_{|z|\to\infty}z{\mathcal H}(z,x)=1\,.
\]
Furthermore,
\begin{equation}\label{thm.stieltjes.claim2}
 {\mathcal G}(z)=\int_{-\pi}^\pi{\mathcal H}(z,x)dx\text{ for all }|z|>\bar R\,.
\end{equation}
\end{enumerate}

\end{theorem}

The last two results of this subsection give neat descriptions of $\mu_f$ in two special cases. In what follows, $WSL(\gamma)$ for $\gamma>0$ denotes the Wigner semicircle law with variance $\gamma$, that is, it is the law whose density is
\begin{equation}\label{eq.defwsl}
 \frac1{2\pi\sqrt\gamma}\sqrt{4-x^2/\gamma\,}\,\one(|x|\le2\sqrt\gamma)\,.
\end{equation}

\begin{theorem}\label{t3}
 If there exists a function $r$ from $[-\pi,\pi]$ to $[0,\infty)$ such that
 \[
\frac12\left[  f(x,y)+f(y,x)\right]=r(x)r(y)\text{ for almost all }x,y\in[-\pi,\pi]\,,
 \]
then
\[
 \mu_f=\eta_r\boxtimes WSL(1)\,,
\]
where $\eta_r$ denotes the law of $2^{3/2}\pi r(U)$, $U$ is a $ Uniform(-\pi,\pi)$ random variable, and `` $\boxtimes$'' denotes the free multiplicative convolution.
\end{theorem}

\begin{theorem}\label{t4}
 Define
\[
 d_{j,k}:=\frac1{2\sqrt2\,\pi}\int_{[-\pi,\pi]^2} e^{-\io(jx+ky)}\sqrt{f(x,y)+f(y,x)}dxdy,\,j,k\in\bbz\,.
\]
Suppose for some increasing sequence of finite subsets $A_1A_2,\ldots$ of $\bbz$ such that $\bigcup_{n\ge1} A_n=\bbz$,\comment{correction-3} 
it holds that\newcomment{Comment 3}
\[
 \sum_{k,l\in\bbz}d_{k,l}d_{j+k,l}\one(k,l,j+k\in A_n)=0\text{ for all }j\in\bbz\setminus\{0\}\text{ and }n\ge1\,.
\]
Then
\[
 \mu_f=WSL(2\|f\|_1)\,.
\]
\end{theorem}

\subsection{The eigen measure}\label{subsec:em}

Theorem \ref{t1} shows that the discrete component of the spectral measure does not have a bearing on the limiting behavior of the $\esd$. Therefore, it is imperative to come up with a variant of the $\esd$ that would capture the role of this component. That end is achieved in this subsection. The first task is to define the proper variant, which we now proceed towards.

It should be remembered that a symmetric matrix always means an $N\times N$ symmetric matrix for some finite $N$.\comment{correction-5}
A symmetric matrix $A$ is to be thought of as a Hermitian operator $\ol A$ of finite rank acting on the first $N$ coordinates of $l^2$, where
\[
 l^p:=\left\{(a_n:n\in\bbn)\subset\bbr:\sum_n|a_n|^p<\infty\right\},\,p\in[1,\infty)\,.
\]
If $\lambda_1\le\ldots\le\lambda_N$ are the eigenvalues of $A$ counted with multiplicity, then the spectrum of $\ol A$ is $\{0,\lambda_1,\ldots,\lambda_N\}$, where $0$ has infinite multiplicity. Motivated by this, we define the {\bf eigen measure} of $A$, denoted by $\eim(A)$, by
\[
 \eim(A):=\infty\delta_0+\sum_{j=1}^N\delta_{\lambda_j}\,.
\]
The measure $\eim(A)$ is to be viewed as an element of the set $\mathcal P$ of point measures $\xi$ of the form
\[
 \xi:=\infty\delta_0+\sum_{j=1}^\infty\delta_{\theta_j}\,,
\]
where $(\theta_j:j\ge1)$ is some sequence of real numbers. It is not hard to see why $\eim(A)$ is an element of $\mathcal P$ for a symmetric matrix $A$ because $\theta_j$ can be taken to be zero after a stage. For $p\in[1,\infty)$, define a subfamily $\mathcal C_p$ of $\mathcal P$ by
\[
 \mathcal C_p:=\left\{\mu\in\mathcal P:\int_\bbr|x|^p\mu(dx)<\infty\right\}\,.
\]
Once again, it is easy to see that for any symmetric matrix $A$,
\[
 \eim(A)\in\mathcal C_p\text{ for all }p\ge1\,.
\]

Fix $p\ge1$ and $\xi\in\mathcal C_p$. Clearly, there exist unique real numbers
\[
 \alpha_1(\xi)\ge\alpha_2(\xi)\ge\ldots\ge0\,,
\]
and
\[
 \alpha_{-1}(\xi)\le\alpha_{-2}(\xi)\le\ldots\le0\,,
\]
such that
\[
 \xi=\infty\delta_0+\sum_{j\neq0}\delta_{\alpha_j(\xi)}\,,
\]
where $\sum_{j\neq0}$ means the sum over all non-zero integers. Define
\[
 d_p(\xi_1,\xi_2):=\left[\sum_{j\neq0}|\alpha_j(\xi_1)-\alpha_j(\xi_2)|^p\right]^{1/p},\,\xi_1,\xi_2\in\mathcal C_p\,.
\]
Given the natural bijection between $\mathcal C_p$ and $l^p$, it is immediate that $(\mathcal C_p,d_p)$ is a complete metric space. Note that
\begin{equation}\label{discrete.eq.bound}
 \left|\left[\int_\bbr |x|^p\xi_1(dx)\right]^{1/p}-\left[\int_\bbr |x|^p\xi_2(dx)\right]^{1/p}\right|\le d_p(\xi_1,\xi_2),\,\xi_1,\xi_2\in\mathcal C_p\,.
\end{equation}
\comment{correction-4} 
It is worth mentioning in this context that the metric defined above is very similar to ``$\delta_2$'' studied in \cite{koltchinskii:gine:2000} and \cite{adamczak:bednorz:2015}.

The main result of this subsection is the following.

\begin{theorem}\label{discrete.t1}
Under the assumption \eqref{eq.assume}, there exists a random point measure $\xi$ which is almost surely in ${\mathcal C}_2$ such that
\begin{equation}\label{discrete.t1.eq4}
 d_4\left(\eim(W_N/N),\xi\right)\prob0\,,
\end{equation}
as $N\to\infty$, where $W_N$ is as defined in \eqref{eq.defan}. Furthermore, the distribution of $\xi$ is determined by $\nu_d$.
\end{theorem}

\begin{remark}
 It is trivial to see that ${\mathcal C}_2\subset{\mathcal C}_4$, and hence one can talk  about the $d_4$ distance between two point measures in ${\mathcal C}_2$.
\end{remark}

\begin{remark}
 There is a notion of convergence different from that in \eqref{discrete.t1.eq4}, namely ``vague convergence''. Suppose that $(\xi_n:1\le n\le\infty)$ are measures on $\bbr$ such that
\[
 \xi_n(\bbr\setminus(-\vep,\vep))<\infty\text{ for all }\vep>0,1\le n\le\infty\,.
\]
Then $\xi_n$ converges vaguely to $\xi_\infty$ if for all $x<0<y$ with $\xi_\infty(\{x,y\})=0$, it holds that
\[
 \lim_{n\to\infty}\xi_n(\bbr\setminus(x,y))=\xi_\infty(\bbr\setminus(x,y))\,.
\]
The vague convergence defined above is same as the vague convergence on $[-\infty,\infty]\setminus\{0\}$ discussed on page 171 in \cite{resnick:2007}, for example. It can be proved without much difficulty that if $(\xi_n:1\le n\le\infty)\subset\mathcal C_p$ for some $p$ such that
\begin{equation}\label{eq.conv}
 \lim_{n\to\infty}d_p(\xi_n,\xi_\infty)=0\,,
\end{equation}
then $\xi_n$ converges to $\xi_\infty$ vaguely. The converse is, however, not true, that is, \eqref{eq.conv} is strictly stronger than vague convergence.
\end{remark}

If Theorem \ref{discrete.t1} is seen as an analogue of Theorem \ref{t1}, then the next natural question should be the analogue of that answered in Theorem \ref{t6}, namely whether $\xi$ restricted to $\bbr\setminus\{0\}$ is non-null and necessarily random. Both these questions are answered in the affirmative in the case when $\nu_d((-\pi,\pi]^2)>0$ by the following result.

\begin{theorem}\label{discrete.t2}
 If $\nu_d((-\pi,\pi]^2)>0$, then the random variable
\[
 \int_\bbr x^2\xi(dx)
\]
is positive almost surely, and non-degenerate, that is, it has a positive variance.\comment{correction-6}
\end{theorem}

\subsection{Long range dependence}\label{subsec:lrd}
In this subsection, we make the connection between the random matrix models and the long range dependence mentioned in Section \ref{sec:intro}. Recalling the fact that for a family of i.i.d. Gaussian random variables, the spectral measure is absolutely continuous, Theorem \ref{t1} can be interpreted as a result about the ``short range dependent'' component of the process $\{X_{j,k}:j,k\in\bbz\}$. Indeed, the LSD $\mu_f$ is completely determined by the absolutely continuous component of the spectral measure. 

On the other hand, Theorem \ref{discrete.t1} establishes the connection between the discrete component of the spectral measure and the limiting eigen measure $\xi$. In the presence of atoms in the spectral measure, a stationary Gaussian process is considered to have a long memory for several reasons. For example, in that case, the process is non-ergodic; see \cite{cornfeld:fomin:sinai:1982}. A trivial example of such a process is the following. Let $G$ be a $N(0,1)$ random variable, and set $X_{j,k}:=G$ for all $j,k$. 

It is also worth noting that in addition to the transition from $\esd$ to $\eim$, the scaling also changes from $\sqrt N$ to $N$ when passing from the former result to the latter.  Therefore, it is  clear that the absolutely continuous and discrete components of the process contribute only towards the LSD and the limiting eigen measure of $W_N$ respectively, albeit with  different scalings. The above observation suggests naturally the following definition of short and long range dependence.

\begin{defn}
 A mean zero stationary Gaussian process with positive variance indexed by $\bbz^2$ is short range dependent if the corresponding spectral measure is absolutely continuous, and the same is long range dependent if the spectral measure is discrete, that is, supported on a countable set.
\end{defn}

The above definitions, of course, are not exhaustive in that there may be processes whose range of dependence is neither short nor long. That can be hoped to be resolved partially if the role of the component $\nu_{cs}$ is understood. This we leave aside for future research.

We conclude this discussion by pointing out that there are other contexts in which long and short range dependence is defined based on absolute continuity of the spectral measure. For example, Section 5 of \cite{samorodnitsky:2006} approaches long range dependence for a stationary second order process indexed by $\bbz$ from the point of view of the growth rate of the variance of its partial sums. In particular, the definition given in (5.14) on page 194 therein is close to the definition given above, though not exactly the same.

\section{A corollary and examples}\label{example}

In this section, a corollary and a few numerical examples that follow from the results of the previous sections are discussed. The first one is a corollary of Theorem \ref{t4}, followed by a numerical example of the same result.

\begin{corollary}
 Assume that $(G_n:n\in\bbz)$ is a {one}-dimensional stationary Gaussian process with zero mean and positive variance, and whose spectral measure is absolutely continuous. Let $((G_{in}:n\in\bbz):i\in\bbz)$ be a family of i.i.d.\ copies of $(G_n:n\in\bbz)$. Define
\[
 X_{j,k}:=G_{j-k,k},\,j,k\in\bbz\,.
\]
Then, $(X_{j,k}:j,k\in\bbz)$ is a stationary Gaussian process, and 
\[
 \mu_f=WSL(2\Var(G_0))\,.
\]
\end{corollary}

\begin{proof}
The hypotheses imply the existence of a non-negative function $h$ on $(-\pi,\pi]$ such that
\[
 \E\left[G_0G_v\right]=\int_{-\pi}^\pi e^{\io vx}h(x)dx,\,v\in\bbz\,.
\]
Clearly, for all $j,k,u,v\in\bbz$,
\begin{eqnarray*}
 \E\left[X_{j,k}X_{j+u,k+v}\right]&=&\E\left[G_0G_v\right]\one(u=v)\,,
\end{eqnarray*}
 which shows the stationarity. Extend $h$ to whole of $\bbr$ by the identity $h(\cdot)\equiv h(\cdot+2\pi)$. Notice that
\begin{eqnarray*}
 &&\int_{[-\pi,\pi]^2}e^{\io(ux+vy)}h\left({x+y}\right)dxdy\\
 &=&\int_{-\pi}^{\pi}e^{\io(u-v)x} \left[\int_{x-\pi}^{x+\pi}e^{\io vz}h(z)dz\right]dx\\
 &=&2\pi\E\left[G_0G_v\right]\one(u=v)\,.
\end{eqnarray*}
\comment{correction-7}
Thus, 
\[
 f(x,y):=(2\pi)^{-1}h(x+y),\,-\pi\le x,y\le\pi\,,
\]
is the spectral density for $(X_{j,k})$. Furthermore, for integers $j\neq k$,
\[
 \int_{[-\pi,\pi]^2}e^{-\io(jx+ky)}\sqrt{f(x,y)}dxdy=0\,,
\]
and therefore, the hypothesis of Theorem \ref{t4} is satisfied with\\ $A_n:=\{-n,\ldots,n\}$. This completes the proof.
\end{proof}

\begin{remark}
The above corollary is false without the assumption that the process $(G_n)$ has a spectral density. For example, if $G_m=G_n$ for all $m,n$, then the matrix $W_N$ becomes a Toeplitz matrix. \cite{BDJ2006} have shown that the LSD has unbounded support in this case.
\end{remark}

\begin{example}\label{example1}
Let $(G_n:n\in\bbz)$ be a zero mean stationary Gaussian process with spectral density $|x|^{-1/2}$, and let $X_{j,k}$ be as in the above corollary. Then, it follows that
\[
 \mu_f=WSL(8\sqrt\pi)\,,
\]
where
\[
 f(x,y):=(2\pi)^{-1}h(x+y),\,-\pi\le x,y\le\pi\,,
\]
with $h(\cdot)$ defined on $\bbr$ by the identities $h(\cdot)=h(\cdot+2\pi)$ and $h(z)=|z|^{-1/2}$ for $-\pi<z\le\pi$. It is easy to see that $\|f\|_2=\infty$, thus showing that the converses of Theorem \ref{t2} and the first part of Theorem \ref{theorem.new} are false.
\end{example}

Next, we shall see two numerical examples where Theorem \ref{t3} holds.\comment{correction-8}

\begin{example}\label{example2}
Let 
\[
 f(x,y)=\one(-\pi/2\le x,y\le\pi/2),\,-\pi\le x,y\le\pi\,.
\]
By Theorem \ref{t3}, it follows that
\[
 \mu_f=\eta_r\boxtimes WSL(1)\,,
\]
where $\eta_r$ is the law of $2^{3/2}\pi\one(|U|\le\pi/2)$, $U$ being a  random variable with the uniform distribution on the interval $(-\pi,\, \pi)$.\comment{correction-9(a)} 
A calculation of the moments of the right hand side using Theorem 14.4 in \cite{nica:speicher:2006} will show that $\mu_f$ is the law of $2\pi BW$ where $B$ and $W$ are independent (in the classical sense) random variables distributed as Bernoulli $(1/2)$ (that is, takes values $0$ and $1$) and $WSL(1)$ respectively. This is an example where the LSD is not a continuous probability measure.
\end{example}

\begin{example}\label{example3}
 Let 
\[
 f(x,y)=|xy|^{-1/2},\,-\pi\le x,y\le\pi\,.
\]
By Theorem \ref{t3}, it follows that
\[
 \mu_f=\eta_r\boxtimes WSL(1)\,,
\]
where $\eta_r$ is the law of $2^{3/2}\pi|U|^{-1/2}$, $U$ following the uniform distribution on $(-\pi,\pi)$, as before. \comment{correction-9(b)}
Since the second moment of $\eta_r$ is infinite, it follows that 
\[
 \int_\bbr x^4\mu_f(dx)=\infty\,.
\]
\end{example}

\section{Proofs}\label{sec:proof}

\subsection{Proofs of Theorems \ref{t1}-\ref{t4}}
We now proceed towards the proof of Theorem \ref{t1}. The proof is by the classical method of moments. However, as illustrated later by Example \ref{example3}, the moments of the LSD need not be finite. Hence, some work is needed to get around that. In order to enable a smooth reading of the proof, we start with a brief sketch of the same, before going into the technical details.

As a first step, we decompose the process $(X_{i,j})$ into the components corresponding to $\nu_{ac}$ and $\nu_d$, and call them $(Y_{i,j})$ and $(Z_{i,j})$ respectively. Using Fact \ref{fact:linprocess} below, $(Y_{i,j})$ is represented in law as a possibly infinite linear process with i.i.d. inputs. Consider a truncation of the obtained linear process at some finite level, that is, a finite linear process. For the random matrix constructed by this finite linear process, the LSD is computed by the standard method of moments. The effect of the truncation, that is, the distance between the ESDs of the matrices constructed by the truncated and non-truncated linear processes, is controlled using the Hoffman-Wielandt inequality. This enables us to compute the LSD of matrices formed with $(Y_{i,j})$, by letting the truncation level go to infinity. Finally, it is shown that adding $Z_{i,j}$ does not change the LSD. This completes the proof of Theorem \ref{t1}.

Define a map $T$ from $(-\pi,\pi]$ to itself by
\[
 T(x)=-x\one(x<\pi)+\pi\one(x=\pi),\,-\pi<x\le\pi\,.
\]
Since the integral on the right hand side of \eqref{p1.eq1} is real for all $u$ and $v$, it follows that $\nu$ is invariant under the transformation $(x,y)\mapsto(T(x),T(y))$, and in particular
\[
 \nu(\{(x,y)\})=\nu(\{(T(x),T(y))\})\text{ for all }x,y\,.
\]
Since the measure $\nu_d$ is concentrated on a countable set, and 
\[
 \nu_d(\{(x,y)\})=\nu(\{(x,y)\})\text{ for all }x,y\,,
\]
it follows that $\nu_d$ is also invariant under the map $T$. By \eqref{eq.assume}, it follows that $\nu_{ac}$ is also invariant under that map, that is,
\begin{equation}\label{eq.even}
 f(x,y)=f(-x,-y)\text{ for almost all }(x,y)\in[-\pi,\pi]^2\,.
\end{equation}
Therefore, for $k,l\in\bbz$, $c_{k,l}$ defined by
\begin{equation}\label{eq.defc}
c_{k,l}:=(2\pi)^{-1}\int_{[-\pi,\pi]^2} e^{-\io(kx+ly)}\sqrt{f(x,y)}\,dxdy\,,
\end{equation}
is a real number. By Parseval's identity, it follows that
\[
 \sum_{k,l\in\bbz}c_{k,l}^2<\infty\,.
\]
Let $(U_{i,j}:i,j\in\bbz)$ be i.i.d. $N(0,1)$ random variables. Define
\begin{equation}\label{eq.defy}
 Y_{i,j}:=\sum_{k,l\in\bbz}c_{k,l}U_{i-k,j-l},\,i,j\in\bbz\,,
\end{equation}
where the  series converges in the $L^2$ norm.\comment{correction-10} 
An important result, on which the current paper is built, is the following fact, which is analogous to a result in Chapter 6, Section 6.6 of \cite{varadhan:2001}.  \comment{correction-11}
\begin{fact}\label{fact:linprocess}
 The process $(Y_{i,j}:i,j\in\bbz)$ defined in \eqref{eq.defy} is a stationary Gaussian process with
\begin{equation}
\label{fact:linprocess.eq1} \E\left(Y_{i,j}Y_{i+u,j+v}\right)=\int_{[-\pi,\pi]^2} e^{\io(ux+vy)}f(x,y)dx\,dy,\,\text{for all }u,v\in\bbz\,. 
\end{equation}
\end{fact}

Since $\nu_d$ is $T$ invariant, it follows that
\begin{equation}\label{neweq1}
 \nu_d=\sum_{j\ge1}\frac{a_j}2\left(\delta_{(x_j,y_j)}+\delta_{(T(x_j),T(y_j))}\right)\,,
\end{equation}
for some at most countable set  $\{(x_1,y_1),(x_2,y_2),\ldots\}\subset(-\pi,\pi]^2$ and non-negative numbers $a_1,a_2,\ldots$ such that $\sum_ja_j<\infty$. Since some of the $a_j$'s can be zero, we can and do assume without loss of generality that the above set is countably infinite. Let $(V_{i,j}:i=1,2,\,j\ge1)$ be a family of i.i.d. $N(0,1)$ random variables which is independent of the family $(U_{i,j}:i,j\in\bbz)$. Define
\begin{equation}\label{eq.defz}
 Z_{i,j}:=\sum_{k=1}^\infty \sqrt{a_k}\left[V_{1,k}\cos(ix_k+jy_k)+V_{2,k}\sin(ix_k+jy_k)\right],\,i,j\in\bbz\,.
\end{equation}
It can be verified that
\[
\E\left(Z_{i,j}Z_{i+u,j+v}\right)=\int_{(-\pi,\pi]^2}e^{\io(ux+vy)}\nu_d(dx,dy)\,,i,j,u,v\in\bbz\,,
\]
since the integral on the right hand side is real. The above along with \eqref{p1.eq1} and \eqref{fact:linprocess.eq1} imply that \comment{correction-12}
\begin{eqnarray*}
(X_{i,j}:i,j\in\bbz)
&\eid&(Y_{i,j}+Z_{i,j}:i,j\in\bbz)\,.
\end{eqnarray*}
 Therefore, without loss of generality, we assume that
\begin{equation}\label{eq15}
 X_{i,j}=Y_{i,j}+Z_{i,j},\,i,j\in\bbz\,.
\end{equation} 

Fix $n\ge1$, and define 
\begin{eqnarray}\label{eq.defyijn}
 Y_{i,j,n}:=\sum_{k,l=-n}^nc_{k,l}U_{i-k,j-l},\,i,j\in\bbz\,,
\end{eqnarray}
and similarly,
\[
 Z_{i,j,n}:=\sum_{k=1}^n\sqrt{a_k}\left[V_{1,k}\cos(ix_k+jy_k)+V_{2,k}\sin(ix_k+jy_k)\right],\,i,j,n\ge1\,.
\]
Set
\begin{equation}\label{eq.defhatf}
\hat f_n(u,v):=\E\left[Y_{i,j,n}Y_{i+u,j+v,n}\right] 
\end{equation}
\[
=\sum_{k,l\in\bbz}c_{k,l}c_{k+u,l+v}\one(|k|\vee|l|\vee|k+u|\vee|l+v|\le n)\,,
\]
for all $u,v\in\bbz$.
For $N,n\ge1$, define the following $N\times N$ symmetric matrices:
\begin{eqnarray}
W_{N,n}(i,j)&:=&Y_{i,j,n}+Y_{j,i,n}\,,\label{eq.defwnn}\\
W_{N,\infty}(i,j)&:=&Y_{i,j}+Y_{j,i}\,,\label{eq.defwninfty}\\
\ol W_{N,n}(i,j)&:=&Y_{i,j,n}+Y_{j,i,n}+Z_{i,j,n}+Z_{j,i,n}\,,\label{eq.defolwnn}\\
\widetilde W_N(i,j)&:=&Z_{i,j}+Z_{j,i}\,,\label{eq.deftildew}\\
\widetilde W_{N,n}(i,j)&:=&Z_{i,j,n}+Z_{j,i,n}\,,\label{eq.deftildewnn}
\end{eqnarray}
for all $1\le i,j\le N$.

Fix $m\ge1$, and $\sigma\in
NC_2(2m)$. Let $(V_1,\ldots,V_{m+1})$ denote the Kreweras
complement of $\sigma$. For $1\le i\le m+1$,
denote
\begin{equation}\label{eq.defV}
 V_i:=\{v_1^i,\ldots,v_{l_i}^i\}\,.
\end{equation}
Define
\begin{equation}\label{eq.defS}
 S(\sigma):=\left\{(k_1,\ldots,k_{2m})\in\bbz^{2m}:\sum_{j=1}^{l_s}k_{v^s_j}=0,\,s=1,\ldots,m+1\right\}\,,
\end{equation}
and
\begin{equation}\label{eq.defbeta}
 \beta_{n,2m}:=\sum_{\sigma\in NC_2(2m)}\sum_{k\in S(\sigma)}\prod_{(u,v)\in\sigma}\left[\hat f_n(k_u,-k_v)+\hat f_n(k_v,-k_u)\right],\,m,n\ge1\,.
\end{equation}
 Notice that even though the set $S(\sigma)$ has infinite cardinality, only finitely many summands on the right hand side above are non-zero, because $\hat f_n(u,v)$ is $0$ if $|u|\vee|v|>2n$.

 Our first step towards proving Theorem \ref{t1} is the following proposition.
 
\begin{prop}\label{ac.p1}
 For fixed $n\ge1$, there exists a compactly supported symmetric probability measure $\mu_{f,n}$ whose $2m$-th moment is $\beta_{n,2m}$ for all $m\ge1$. Furthermore,
 \[
  \esd(W_{N,n}/\sqrt N)\to\mu_{f,n}\,,
 \]
weakly in probability, as $N\to\infty$.
\end{prop}

For the proof of the above proposition, a few combinatorial notions will be required. Fix $m,n\ge1$. 
\begin{defn}
For $\sigma\in NC_2(2m)$, a tuple $j:=(j_1,\ldots,j_{2m})\in\bbz^{2m}$ is \textbf{$\sigma$-Catalan}, if 
\[
(j_{u-1},j_u)=(j_v,j_{v-1})\text{ for all }(u,v)\in\sigma\,,
\]
with the convention that $j_0:=j_{2m}$, a convention that will be used throughout the paper.
\end{defn}

\begin{lemma}\label{comb.l1}
Fix $\sigma\in NC_2(2m)$ and $N\ge1$. 
Label the blocks of the Kreweras complement of $\sigma$ and the elements thereof as in \eqref{eq.defV} and \eqref{eq.defS}, with the additional requirement that
\comment{correction-13(b)}
\begin{equation}\label{comb.l1.eq3}
v_1^u\le\ldots\le v_{l_u}^u\,,1\le u\le m+1\,.
\end{equation}
 Assume that $i\in\{1,\ldots,N\}^{2m}$ satisfies for all $1\le u<v\le2m$,
\begin{equation}\label{comb.l1.eq1}
\left|i_{u-1}-i_v\right|\vee\left|i_u-i_{v-1}\right|\le 2n+1\text{ if }(u,v)\in\sigma\,,
\end{equation}
and
\begin{equation}\label{comb.l1.eq2}
|i_u-i_v|>9mn\text{ whenever }{\mathcal T}_\sigma(u)\neq{\mathcal T}_\sigma(v)\,,
\end{equation}
where ${\mathcal T}_\sigma(\cdot)$ is as in \eqref{eq.defTsigma}. Then there exist a $\sigma$-Catalan tuple $j$ and a $k\in S(\sigma)$ such that\newcomment{Comment 6}
\begin{equation}\label{comb.l1.eq4}
i_{v_x^u}=j_{v_1^u}+\sum_{h=1}^xk_{v_h^u}\,,1\le x\le l_u\,,1\le u\le m+1\,.
\end{equation}
Furthermore,
\begin{equation}\label{comb.l1.eq5}
\E\left(\prod_{u=1}^{2m}W_{N,n}(i_{u-1},i_u)\right)=\prod_{(w,x)\in\sigma}\left[\hat f_n(k_w,-k_x)+\hat f_n(k_x,-k_w)\right]\,,
\end{equation}\newcomment{Comment 4}
$\hat f(\cdot,\cdot)$ being as in \eqref{eq.defhatf}.\comment{correction-13(c)}
\end{lemma}

\begin{proof}
Fix $i\in\{1,\ldots,N\}^{2m}$ satisfying \eqref{comb.l1.eq1} and \eqref{comb.l1.eq2}. Define a tuple $k\in\bbz^{2m}$ by
\begin{equation}\label{comb.l1.eq6}
k_{v^u_h}:=
\begin{cases}
i_{v^u_{h+1}}-i_{v^u_h},&1\le h\le l_u-1\,,1\le u\le m+1\,,\\
i_{v^u_1}-i_{v^u_{l_u}},&h=l_u\,,1\le u\le m+1\,.
\end{cases}
\end{equation}
It is immediate that
\[
\sum_{h=1}^{l_u}k_{v^u_h}=0\text{ for all }1\le u\le m+1\,,
\]
and hence $k\in S(\sigma)$. Setting
\[
j_{v^u_h}:=i_{v^u_{l_u}}\,,1\le h\le l_u\,,1\le u\le m+1\,,
\]
it is easy to see that $j$ is $\sigma$-Catalan, and that \eqref{comb.l1.eq4} holds. 

In order to show \eqref{comb.l1.eq5}, we fix $(w,x)\in\sigma$ and claim that
\begin{eqnarray}
i_w-i_{x-1}&=&k_w\,,\label{comb.l1.eq7}\\
i_x-i_{w-1}&=&k_x\,.\label{comb.l1.eq8}
\end{eqnarray}
Assume without loss of generality that $w<x$. Therefore, $\overline w$ and $\overline{x-1}$ belong to the same block in $K(\sigma)$, and furthermore, the block containing them is a subset of $\{\overline w,\overline{w+1},\ldots,\overline{x-1}\}$. Thus, \eqref{comb.l1.eq7} follows from \eqref{comb.l1.eq4}. We show \eqref{comb.l1.eq8} separately for the cases $w\ge2$ and $w=1$. If $w\ge2$, then $\overline{w-1}$ and $\overline x$ are in the same block of $K(\sigma)$, and furthermore that block does not intersect with $\{\overline w,\overline{w+1},\ldots,\overline{x-1}\}$. This shows \eqref{comb.l1.eq8}, once again with the help of \eqref{comb.l1.eq4}. If $w=1$, then $\overline{2m}$ and $\overline x$ are in the same block of $K(\sigma)$. Obviously, $\overline{2m}$ has to be the last member of its block. Since $(1,x)\in\sigma$, it follows that $\overline x$ is the first member of the block containing itself and $\overline{2m}$, showing that
\[
 i_x=i_{2m}+k_x=i_0+k_x=i_{w-1}+k_x\,.
\]
This completes the proof of \eqref{comb.l1.eq8}

Our next claim is that 
\begin{equation}\label{comb.l1.eq9}
\E\left(W_{N,n}(i_{w-1},i_w)W_{N,n}(i_{x-1},i_x)\right)=\hat f_n(k_w,-k_x)+\hat f_n(k_x,-k_w)\,.
\end{equation}
Since $(w,x)\in\sigma$, it is immediate that ${\mathcal T}_\sigma(w)\neq{\mathcal T}_\sigma(x)$.
The inequality \eqref{comb.l1.eq2} implies that
\[
|i_w-i_x|>9mn>2n+1\,,
\]
and hence
\[
\E\left(Y_{i_{w-1},i_w,n}Y_{i_{x-1},i_x,n}\right)=\E\left(Y_{i_w,i_{w-1},n}Y_{i_x,i_{x-1},n}\right)=0\,.
\]
This  implies that 
\begin{eqnarray*}
&&\E\left(W_{N,n}(i_{w-1},i_w)W_{N,n}(i_{x-1},i_x)\right)\\
&=&\E\left(Y_{i_{w-1},i_w,n}Y_{i_x,i_{x-1},n}\right)+\E\left(Y_{i_w,i_{w-1},n}Y_{i_{x-1},i_x,n}\right)\\
&=&\hat f_n(i_x-i_{w-1},i_{x-1}-i_w)+\hat f_n(i_w-i_{x-1},i_{w-1}-i_x)\\
&=&\hat f_n(k_x,-k_w)+\hat f_n(k_w,-k_x)\,,
\end{eqnarray*}
the last line following from \eqref{comb.l1.eq7} and \eqref{comb.l1.eq8}, thereby establishing \eqref{comb.l1.eq9}.

The claim \eqref{comb.l1.eq5} will follow if it can be shown that $\xi_1,\ldots,\xi_m$ are independent, where
\[
\xi_j:=W_{N,n}(i_{w_j-1},i_{w_j})W_{N,n}(i_{x_j-1},i_{x_j})\,,1\le j\le m\,,
\]
and $\sigma=\{(w_1,x_1),\ldots,(w_m,x_m)\}$.  Fix $1\le j<k\le m$, and since $\sigma$ is non-crossing, it can be assumed without loss of generality that either
\begin{equation}
\label{neweq.opt1}w_j<x_j<w_k<x_k\,,
\end{equation}
or\newcomment{Comment 7}
\begin{equation}
\label{neweq.opt2}w_j<w_k<x_k<x_j\,.
\end{equation}
Since the entries of $W_{N,n}$ are jointly Gaussian, for showing independence, it suffices to check that the covariances vanish.
In view of \eqref{comb.l1.eq2}, it follows that 
\[
\E\left(W_{N,n}(i_w,i_x)W_{N,n}(i_y,i_z)\right)=0\text{ if }{\mathcal T}_\sigma(y)\neq{\mathcal T}_\sigma(x)\neq{\mathcal T}_\sigma(z)\,.
\]

Suppose that \eqref{neweq.opt1} holds.  Notice that
\begin{eqnarray}
\label{rev2.neweq1}\E\left(W_{N,n}(i_{w_j-1},i_{w_j})W_{N,n}(i_{w_k-1},i_{w_k})\right)&=&0\\
\nonumber\text{because }{\mathcal T}_\sigma(w_j)\neq{\mathcal T}_\sigma(w_k)\neq{\mathcal T}_\sigma(w_j-1)\,,&&\\
\label{rev2.neweq2}\E\left(W_{N,n}(i_{w_j-1},i_{w_j})W_{N,n}(i_{x_k-1},i_{x_k})\right)&=&0\\
\nonumber\text{because }{\mathcal T}_\sigma(x_k-1)\neq{\mathcal T}_\sigma(w_j)\neq{\mathcal T}_\sigma(x_k)\,,&&\\
\label{rev2.neweq3}\E\left(W_{N,n}(i_{x_j-1},i_{x_j})W_{N,n}(i_{w_k-1},i_{w_k})\right)&=&0\\
\nonumber\text{because }{\mathcal T}_\sigma(x_j)\neq{\mathcal T}_\sigma(w_k)\neq{\mathcal T}_\sigma(x_j-1)\,,&&\\
\label{rev2.neweq4}\E\left(W_{N,n}(i_{x_j-1},i_{x_j})W_{N,n}(i_{x_k-1},i_{x_k})\right)&=&0\\
\nonumber\text{because }{\mathcal T}_\sigma(x_j)\neq{\mathcal T}_\sigma(x_k-1)\neq{\mathcal T}_\sigma(x_j-1)\,.&&
\end{eqnarray}

Next, suppose that \eqref{neweq.opt2} holds. The equalities \eqref{rev2.neweq1} and \eqref{rev2.neweq4} will still hold for the same reasons as given above. We need to check \eqref{rev2.neweq2} and \eqref{rev2.neweq3}. To that end, observe that
\begin{eqnarray*}
\E\left(W_{N,n}(i_{w_j-1},i_{w_j})W_{N,n}(i_{x_k-1},i_{x_k})\right)&=&0\\
\nonumber\text{because }{\mathcal T}_\sigma(w_j-1)\neq{\mathcal T}_\sigma(x_k-1)\neq{\mathcal T}_\sigma(w_j)\,,&&\\
\E\left(W_{N,n}(i_{x_j-1},i_{x_j})W_{N,n}(i_{w_k-1},i_{w_k})\right)&=&0\\
\nonumber\text{because }{\mathcal T}_\sigma(w_k-1)\neq{\mathcal T}_\sigma(w_k)\neq{\mathcal T}_\sigma(x_j)\,.&&
\end{eqnarray*}

Thus, $\xi_j$ and $\xi_k$ are independent. In fact, it follows that $\xi_1,\ldots,\xi_m$ are independent which along with \eqref{comb.l1.eq9}, establishes \eqref{comb.l1.eq5}, and thereby completes the proof.
\end{proof}

\begin{defn}\label{defn.cat}
For $N,m,n\ge1$, $\sigma\in NC_2(2m)$ and $k\in S(\sigma)$, let\\
 $Cat(N,n,k,\sigma)$ denote the set of tuples $i\in\{1,\ldots,N\}^{2m}$ such that \eqref{comb.l1.eq1} - \eqref{comb.l1.eq4} hold.
\end{defn}

\begin{defn}\label{defn.pm} Let ${\mathcal P}(2m)$ denote the set of all pair partitions of $\{1,\ldots,2m\}$ (which may or may not be crossing).
For $N,n\ge1$, let $PM(N,n)$ denote the set of all tuples $i\in\{1,\ldots,N\}^{2m}$ such that there exists a $\pi\in{\mathcal P}(2m)$ satisfying
\begin{equation}
\label{defn.pm.eq1}\left|i_{u-1}\wedge i_u-i_{v-1}\wedge i_v\right|\vee\left|i_{u-1}\vee i_u-i_{v-1}\vee i_v\right|\le 2n+1\,,
\end{equation}
for all $(u,v)\in\pi$.
\end{defn}

\begin{remark} For all $N,m,n\ge1$, $\sigma\in NC_2(2m)$ and $k\in S(\sigma)$, it is clear that\newcomment{Comment 8}
\[
Cat(N,n,k,\sigma)\subset PM(N,n)\,.
\]
Equations \eqref{comb.l1.eq1}, \eqref{comb.l1.eq7} and \eqref{comb.l1.eq8} show that if $k\in S(\sigma)$ is such that $\max_j|k_j|>2n+1$, then $Cat(N,n,k,\sigma)$ is empty.
Furthermore, for any $\sigma_1,\sigma_2\in NC_2(2m)$,
\[
S(\sigma_1)\cap S(\sigma_2)\neq\emptyset\,,
\]
because the zero tuple belongs to both. However, if $k_t\in S(\sigma_t)$ for $t=1,2$, then
\[
Cat(N,n,k_1,\sigma_1)\cap Cat(N,n,k_2,\sigma_2)=\emptyset\,,
\]
whenever either $\sigma_1\neq\sigma_2$ or $k_1\neq k_2$.
\end{remark}

The following combinatorial result will also be used for the proof of Proposition \ref{ac.p1}.

\begin{lemma}\label{comb.l2}
For all fixed $m,n\ge1$,\newcomment{Comment 10}
\[
\lim_{N\to\infty}N^{-(m+1)}\#\left(PM(N,n)\setminus\bigcup_{\sigma\in NC_2(2m)}\bigcup_{k\in S(\sigma)}Cat(N,n,\sigma,k)\right)=0\,,
\]
and for all fixed $\sigma\in NC_2(2m)$ and $k\in S(\sigma)$,
\[
\lim_{N\to\infty}N^{-(m+1)}\#Cat(N,n,\sigma,k)=1\,.
\]
\end{lemma}

\begin{proof}\newcomment{Comment 9}
Fix $m,N\ge1$ and $\sigma\in NC_2(2m)$, and for a moment instead of $Cat(N,n,k,\sigma)$, consider the following set:
\[
C(N,\sigma):=\left\{i\in\{1,\ldots,N\}^{2m}:i_{u-1}=i_v,\,i_u=i_{v-1}\text{ for all }(u,v)\in\sigma\right\}\,.
\]
For all $N\ge1,n\in\{-1/2\}\cup\bbn$ and $\pi\in{\mathcal P}(2m)$, define the set
\[
P(N,n,\pi):=\bigl\{i\in\{1,\ldots,N\}^{2m}:\left|i_{u-1}\wedge i_u-i_{v-1}\wedge i_v\right|\vee
\]
\[
\,\,\,\,\,\,\,\,\,\left|i_{u-1}\vee i_u-i_{v-1}\vee i_v\right|\le2n+1\text{ for all }(u,v)\in\pi\bigr\}\,.
\]
Note that for $\sigma\in NC_2(2m)$, $P(N,-1/2,\sigma)\supset C(N,\sigma)$.

The following are very well known combinatorial results which have been frequently used in random matrix theory, for example to prove Wigner's theorem; see the proof of Theorem 4 in \cite{bose:sen:2008}. For any $\pi\in{\mathcal P}(2m)$,
\begin{equation}
\label{comb.l2.eq1}\lim_{N\to\infty}N^{-(m+1)}\#P(N,-1/2,\pi)=
\begin{cases}
1\,,&\text{if }\pi\in NC_2(2m)\,,\\
0\,,&\text{otherwise}\,,
\end{cases}
\end{equation}
and for $\sigma\in NC_2(2m)$,
\begin{equation}
\label{comb.l2.eq2}\lim_{N\to\infty}N^{-(m+1)}\#\left(P(N,-1/2,\sigma)\setminus C(N,\sigma)\right)=0\,.
\end{equation}

Fix $n\ge1$ and $\sigma\in NC_2(2m)$. The claim \eqref{comb.l1.eq4} of Lemma \ref{comb.l1} implies that as $N\to\infty$, for all fixed $k\in S(\sigma)$,
\begin{equation}
\label{comb.l2.eq6}\#Cat(N,n,k,\sigma)=\#C(N,\sigma)+O(1)\,.
\end{equation}
This, in view of \eqref{comb.l2.eq1} and \eqref{comb.l2.eq2}, implies the second claim of the lemma.

For $i\in\bbn^{2m}$, denote
\[
\|i\|_\infty:=i_1\vee\ldots\vee i_{2m}\,,
\]
and 
\[
B_r(i):=\{j\in\bbn^{2m}:\|i-j\|_\infty\le r\}\,,r\ge0\,.
\]
It is easy to see that for all $r\ge0$,
\begin{equation}
\label{comb.l2.eq3}\sup_{i\in\bbn^{2m}}\#B_r(i)<\infty\,.
\end{equation}
Fix $n\ge1$ and $i\in PM(N,n,\pi)$. Clearly, there exists $j\in P(N,-1/2,\pi)$ such that
\[
\|i-j\|_\infty\le8mn\,.
\]
Thus, for all $N\ge1$,
\[
P(N,n,\pi)\subset\bigcup_{i\in P(N,-1/2,\pi)}B_{8mn}(i)\,.
\]
This along with \eqref{comb.l2.eq1} and \eqref{comb.l2.eq3} implies that for all $\pi\in{\mathcal P}(2m)\setminus NC_2(2m)$, 
\begin{equation}
\label{comb.l2.eq4}\lim_{N\to\infty}N^{-(m+1)}\#P(N,n,\pi)=0\,.
\end{equation}
Using \eqref{comb.l2.eq2} and arguments similar to those leading to \eqref{comb.l2.eq6}, one can show that for all $\sigma\in NC_2(2m)$,
\begin{equation}
\label{comb.l2.eq5}\lim_{N\to\infty}N^{-(m+1)}\#\left(P(N,n,\sigma)\setminus\bigcup_{k\in S(\sigma)}Cat(N,n,k,\sigma)\right)=0\,.
\end{equation}
The observation that
\[
PM(N,n)=\bigcup_{\pi\in{\mathcal P}(2m)}P(N,n,\pi)\,,
\]
along with \eqref{comb.l2.eq4} and \eqref{comb.l2.eq5} establish the first claim of the lemma, thus completing the proof.
\end{proof}

\begin{proof}[Proof of Proposition \ref{ac.p1}]
The proof is by the method of moments. As is now standard in the literature, for executing the proof, it is sufficient to show that
\begin{equation}\label{ac.p1.eq1}
 \lim_{N\to\infty}N^{-(m+1)}\E\left[\Tr\left(W_{N,n}^{2m}\right)\right]=\beta_{n,2m}\text{ for all }m\ge1\,,
\end{equation}
\begin{equation}\label{ac.p1.eq2}
 \lim_{N\to\infty}N^{-(m+2)}\Var\left[\Tr\left(W_{N,n}^{m}\right)\right]=0\text{ for all }m\ge1\,,
\end{equation}
and
\begin{equation}\label{ac.p1.eq3}
 \limsup_{m\to\infty}\beta_{n,2m}^{1/2m}<\infty\,.
\end{equation}
It is easy to see from  symmetry that\comment{correction-13(a)}
\[
\E\left[\Tr\left(W_{N,n}^{m}\right)\right]=0\text{ for all }N\ge1\,,\text{ odd }m\,,
\]
justifying why it suffices to consider only the even moments in \eqref{ac.p1.eq1}. The above along with \eqref{ac.p1.eq1} and \eqref{ac.p1.eq2} implies that for all $m\ge1$, as $N\to\infty$,
\[
N^{-1}\Tr\left((W_{N,n}/\sqrt N)^m\right)\prob
\begin{cases}
\beta_{n,m},&m\text{ even}\,,\\
0,&m\text{ odd}\,.
\end{cases}
\]

We start with showing \eqref{ac.p1.eq1}. To that end, fix $m\ge1$, and for $i:=(i_1,\ldots,i_{2m})\in\bbz^{2m}$, define
\begin{equation}\label{ac.p1.eq12}
 E_i:=\E\left[\prod_{j=1}^{2m}\left(W_{N,n}(i_{j-1},i_j)\right)\right]\,,
\end{equation}
with the convention that $i_0:=i_{2m}$ for all $i\in\bbz^{2m}$. Recall that
\[
 \E\left[\Tr\left(W_{N,n}^{2m}\right)\right]=\sum_{i\in\{1,\ldots,N\}^{2m}}E_i\,.
\]
Notice that for all $1\le h,j,k,l\le N$,
\[
\E(Y_{h,j,n}Y_{k,l,n})=0\text{ whenever }|h-k|\vee|j-l|>2n+1\,,
\]
and hence
\begin{equation}
\label{ac.p1.eq13}\E\left(W_{N,n}(h,j)W_{N,n}(k,l)\right)=0\text{ if }|h\wedge j-k\wedge l|\vee|h\vee j-k\vee l|>2n+1\,.
\end{equation}
Since $\{Y_{h,j,n}:h,j\in\bbz\}$ are jointly Gaussian,\comment{correction-14} by the Wick's formula, \newcomment{Comment 11}
see for example \citet[Theorem 1.28]{janson:1997}, it follows that for $i\in\{1,\ldots,N\}^{2m}$,
\begin{equation}\label{ac.p1.wick}
E_i=\sum_{\pi\in{\mathcal P}(2m)}\prod_{(u,v)\in\pi}\E\left[W_{N,n}(i_{u-1},i_u)W_{N,n}(i_{v-1},i_v)\right]\,,
\end{equation}
which is zero if $i\notin PM(N,n)$, by the preceding equation. Thus, it follows that
\begin{equation}\label{ac.p1.eq10}
 \E\left[\Tr\left(W_{N,n}^{2m}\right)\right]=\sum_{i\in PM(N,n)}E_i\,.
\end{equation}
Equation \eqref{ac.p1.wick} implies that 
\begin{equation}\label{ac.p1.eq11}
\sup_{N\ge1,i\in\{1,\ldots,N\}^{2m}}|E_i|<\infty\,.
\end{equation}
This, in view of the first claim of Lemma \ref{comb.l2} and the remark preceding it, implies that 
\begin{equation}\label{ac.p1.eq4}
\E\left[\Tr\left(W_{N,n}^{2m}\right)\right]=o\left(N^{m+1}\right)+\sum_{\sigma\in NC_2(2m)}\sum_{k\in S(\sigma)}\sum_{{i\in Cat(N,n,\sigma,k)}}E_i\,,
\end{equation}
as $N\to\infty$. Lemma \ref{comb.l1} implies that for $i\in Cat(N,n,\sigma,k)$, $E_i$ equals the right hand side of \eqref{comb.l1.eq5}. Hence, the right hand side of \eqref{ac.p1.eq4} becomes
\begin{eqnarray}
\label{ac.p1.eq14}&&o\left(N^{m+1}\right)+\sum_{\sigma\in NC_2(2m)}\sum_{k\in S(\sigma)}\#Cat(N,n,\sigma,k)\prod_{(u,v)\in\sigma}\Bigl[\hat f_n(k_u,-k_v)\\
\nonumber&&\,\,\,\,\,\,\,\,\,\,\,\,\,\,\,\,\,\,\,\,\,\,\,\,+\hat f_n(k_v,-k_u)\Bigr]\\
\nonumber&=&(1+o(1))N^{m+1}\beta_{n,2m}\,,
\end{eqnarray}
the last line following from the second claim of Lemma \ref{comb.l2} and \eqref{eq.defbeta}. This along with \eqref{ac.p1.eq4} establishes \eqref{ac.p1.eq1}. The proof of \eqref{ac.p1.eq2} follows by a similar combinatorial analysis which is analogous to the proof by method of moments for the classical Wigner matrix, and hence is deferred till the Appendix.

The proof will be complete if \eqref{ac.p1.eq3} can be shown. To that end observe that
\[
 \beta_{n,2m}\le\left(32n^2\max_{|u|\vee|v|\le2n}|\hat f_n(u,v)|\right)^m\#NC_2(2m)\,.
\]
It can be shown by Stirling's approximation that
\[
 \#NC_2(2m)=O\left(4^m\right)\,,
\]
and hence \eqref{ac.p1.eq3} follows. This completes the proof.
\end{proof}

Recall the $N\times N$ random matrix $\ol W_{N,n}$ from \eqref{eq.defolwnn}. The second step in the proof of Theorem \ref{t1} is the following lemma.

\begin{lemma}\label{ac.l1}
 For fixed $n\ge1$, as $N\to\infty$,
 \[
  \esd(\ol W_{N,n}/\sqrt N)\to\mu_{f,n}\,,
 \]
weakly in probability, where $\mu_{f,n}$ is as in the statement of Proposition \ref{ac.p1}.
\end{lemma}

For the proof of the above result, we shall use the following fact which follows from Theorem A.43 on page 503 in \cite{bai:silverstein:2010}.

\begin{fact}\label{fact:rank}
 Let us denote by $L$, the L\'evy distance, defined as
\begin{equation*}
 L(\nu_1,\nu_2):=\inf\Bigl\{\vep>0:\nu_1\left((-\infty,x-\vep]\right)-\vep\le\nu_2\left((-\infty,x]\right)\le
\end{equation*}
\[
 \nu_1\left((-\infty,x+\vep]\right)+\vep\mbox{ for all }x\in\bbr\Bigr\}\,,
\]
for probability measures $\nu_1,\nu_2$ on $\bbr$.\comment{correction-15}
For any two $N\times N$ real symmetric matrices $A$ and $B$ we have,
\[
 L\left(\esd(A),\esd(B)\right)\le\frac1N\rank(A-B)\,.
\]
\end{fact}

\begin{proof}[Proof of Lemma \ref{ac.l1}]
 All that needs to be shown is that
\[
 L\left(\esd(\ol W_{N,n}/\sqrt N),\mu_{f,n}\right)\prob0\,,
\]
as $N\to\infty$. In view of Proposition \ref{ac.p1}, it suffices to show that
\[
 L\left(\esd(\ol W_{N,n}/\sqrt N),\esd(W_{N,n}/\sqrt N)\right)\prob0\,,
\]
as $N\to\infty$. To that end, notice that by Fact \ref{fact:rank},
\begin{eqnarray*}
&& L\left(\esd(\ol W_{N,n}/\sqrt N),\esd(W_{N,n}/\sqrt N)\right)\\
&\le&
\frac1N\rank\left(\ol W_{N,n}-W_{N,n}\right)\,.
\end{eqnarray*}
It is easy to see that the rank of the $N\times N$ matrix whose $(i,j)$-th entry is $Z_{i,j,n}$ is at most $4n$. Therefore,
\[
 \rank\left(\ol W_{N,n}-W_{N,n}\right)\le8n\,.
\]
This completes the proof.
\end{proof}

For the final step in the proof of Theorem \ref{t1}, we shall use the following fact which is also well known. For the sake of completeness, a proof is included in the Appendix.

\begin{fact}\label{f4}
 Let $(\Sigma,d)$ be a {complete} metric space, and let $(\Omega,{\mathcal A},P)$ be a probability space. Suppose that $(X_{mn}:(m,n)\in\{1,2,\ldots,\infty\}^2\setminus\{\infty,\infty\})$ is a family of random elements in $\Sigma$, that is, measurable maps from $\Omega$ to $\Sigma$, the latter being equipped with the Borel $\sigma$-field induced by $d$. Assume that 
\begin{enumerate}
\item for all fixed $1\le m<\infty$,
\[
 d(X_{mn},X_{m\infty})\prob0\,,
\]
as $n\to\infty$,
\item and, for all $\vep>0$,
\[
 \lim_{m\to\infty}\limsup_{n\to\infty}P\left[d(X_{mn},X_{\infty n})>\vep\right]=0\,.
\]
\end{enumerate}
Then, there exists a random element $X_{\infty\infty}$ of $\Sigma$ such that
\begin{equation}\label{f4.eq1}
 d(X_{m\infty},X_{\infty\infty})\prob0\,,
\end{equation}
as $m\to\infty$, and
\[
 d(X_{\infty n},X_{\infty\infty})\prob0\,,
\]
as $n\to\infty$. Furthermore, if $X_{m\infty}$ is deterministic for all $m$, then so is $X_{\infty\infty}$, and then \eqref{f4.eq1} simplifies to
\begin{equation}\label{f4.eq2}
 \lim_{m\to\infty}d(X_{m\infty},X_{\infty\infty})=0\,.
\end{equation}
\end{fact}

\begin{proof}[Proof of Theorem \ref{t1}]
 The space of probability measures on $\bbr$ is a complete metric space when equipped with the L\'evy distance $L(\cdot,\cdot)$. In view of Lemma \ref{ac.l1} and Fact \ref{f4}, all that needs to be shown to complete the proof is that
\begin{equation}\label{t1.eq1}
 \lim_{n\to\infty}\limsup_{N\to\infty}P\left(L\left(\esd(W_N/\sqrt N),\esd(\ol W_{N,n}/\sqrt N)\right)>\vep\right)=0\,,
\end{equation}
for all $\vep>0$. Note that the roles of the indices $m$ and $n$ in Fact \ref{f4} are played by $n$ and $N$ respectively in \eqref{t1.eq1}.\comment{correction-16} 
To that end, fix $\vep>0$ and observe that
\begin{eqnarray*}
 &&P\left(L\left(\esd(W_N/\sqrt N),\esd(\ol W_{N,n}/\sqrt N)\right)>\vep\right)\\
 &\le&\vep^{-3}\E\left[L^3\left(\esd(W_N/\sqrt N),\esd(\ol W_{N,n}/\sqrt N)\right)\right]\\
 &\le&\vep^{-3}N^{-2}\E\Tr\left[(W_N-\ol W_{N,n})^2\right]\,,
\end{eqnarray*}
the inequality in the last line following from the Hoffman-Wielandt inequality; see Corollary A.41 on page 502 in \cite{bai:silverstein:2010}. Clearly, by \eqref{eq15}, it follows that
\begin{eqnarray}
&& \E\Tr\left[(W_N-\ol W_{N,n})^2\right]\label{ac.t1.eq1}\\
&\le&4\sum_{i,j=1}^N\left[\E\left[(Y_{i,j}-Y_{i,j,n})^2\right]+\E\left[(Z_{i,j}-Z_{i,j,n})^2\right]\right]\nonumber\\
&=&4N^2\left[\sum_{k=n+1}^\infty a_k+\sum_{i,j\in\bbz:|i|\vee|j|>n}c_{i,j}^2\right]\,.\label{ac.t1.eq2}
\end{eqnarray}
Using the assumptions $\sum_j a_j<\infty$ and $\sum_{j,k\in \bbz} c_{j,k}^2<\infty$, it follows that the term inside the bracket in \eqref{ac.t1.eq2} goes to zero and hence \eqref{t1.eq1} follows. \comment{correction-17} 
Fact \ref{f4} ensures the existence of a deterministic probability measure $\mu_f$ such that 
\[
 L\left(\esd(W_N/\sqrt N),\mu_f\right)\prob0\,,
\]
as $N\to\infty$. 

Furthermore, assertion \eqref{f4.eq2} ensures that
\begin{equation}\label{t1.eq2}
 \mu_{f,n}\weak\mu_f\text{ as }n\to\infty\,.
\end{equation}
From the definition, it is easy to see that $\mu_{f,n}$ is determined by $f$ for every $n\ge1$, and hence so is $\mu_f$.
This completes the proof of Theorem \ref{t1}.
\end{proof}

\begin{remark}\label{rem:muf}
Since $\mu_{f,n}$ is compactly supported for each $n$, its characteristic function is\comment{correction-18}
\[
 \int_\bbr e^{itx}\mu_{f,n}(dx)=1+\sum_{m=1}^\infty(-1)^m\frac{\beta_{n,2m}}{(2m)!}t^{2m},\,t\in\bbr\,.
\]
Thus, the characteristic function of $\mu_f$ is
\[
 \int_\bbr e^{itx}\mu_{f}(dx)=1+\lim_{n\to\infty}\sum_{m=1}^\infty(-1)^m\frac{\beta_{n,2m}}{(2m)!}t^{2m},\,t\in\bbr\,.
\]
It is worth noting that exchanging the sum and limit above does not make sense because $\lim_{n\to\infty}\beta_{n,2m}$ may or may not be finite. Example \ref{example3} is one where the limit is infinite for all $m\ge2$.
\end{remark}

\begin{proof}[Proof of Theorem \ref{t6}]
 Denote the probability space on which we were working so far by $(\Omega,{\mathcal A},P)$. In particular, the random matrices $ W_{N,n}$  are defined on this probability space. Consider the interval $(0,1)$ equipped with the standard Borel $\sigma$-field $\mathcal B((0,1))$ and the Lebesgue measure $Leb$ which when restricted to $(0,1)$ becomes a probability measure. Define a master probability space
\[
 \left(\Omega\times(0,1),\mathcal A\times\mathcal B((0,1)),\mathbb P:=P\times Leb\right)\,.
\]
Denote the expectation with respect to $\mathbb P$ by $\mathbb E$. 
By Proposition \ref{ac.p1} and the Cantor diagonalization principle, one can choose positive integers $N_1<N_2<N_3<\ldots$ such that for all fixed $n\ge1$,
\[
 \esd( W_{N_k,n}/\sqrt {N_k})\to\mu_{f,n}\text{ as }k\to\infty\,,
\]
weakly {\bf almost surely}, that is,
\begin{equation}\label{t6.eq1}
 \lim_{k\to\infty}L\left(\esd( W_{N_k,n}/\sqrt {N_k}),\mu_{f,n}\right)=0\text{ almost surely},
\end{equation}
for all fixed $n\ge1$, where $L$ is the L\'evy distance. For $1\le k,n<\infty$, we define  random variables\comment{correction-19} $\chi_{k,n}$ on $\Omega\times(0,1)$ by
\[
 \chi_{k,n}(\omega,x):=N_k^{-1/2}\lambda_{\lceil N_kx\rceil}\left( W_{N_k,n}(\omega)\right),\,\omega\in\Omega,x\in(0,1)\,.
\]
Furthermore, for all $k$, define
\[
 \chi_{k,\infty}(\omega,x):=N_k^{-1/2}\lambda_{\lceil N_kx\rceil}\left( W_{N_k,\infty}(\omega)\right),\,\omega\in\Omega,x\in(0,1)\,,
\]
where $W_{N,\infty}$ is as in \eqref{eq.defwninfty}.
Finally, for all $n\ge1$, define
\[
 \chi_{\infty,n}(\omega,x):=F_n^\leftarrow(x),\,\omega\in\Omega,x\in(0,1)\,,
\]
where $F_n(\cdot)$ is the c.d.f. corresponding to $\mu_{f,n}$, and for any c.d.f. $F(\cdot)$, $F^\leftarrow(\cdot)$ is defined by
\[
 F^\leftarrow(y):=\inf\left\{x\in\bbr:F(x)\ge y\right\},\,0<y<1\,.
\]
Our first goal is to show that for all fixed $1\le n<\infty$,
\begin{equation}\label{t6.eq2}
 \chi_{k,n}\to\chi_{\infty,n}\,\mathbb P\text{-almost surely, as }k\to\infty\,.
\end{equation}
To that end, define the set
\[
 A:=\left\{\omega\in\Omega:\lim_{k\to\infty}L\left(\esd( W_{N_k,n}(\omega)/\sqrt {N_k}),\mu_{f,n}\right)=0\text{ for all }n\ge1\right\}\,.
\]
By \eqref{t6.eq1}, it follows that $P(A)=1$. Therefore, for establishing \eqref{t6.eq2}, it suffices to show that for all $\omega\in A$,
\begin{equation}\label{t6.eq3}
 \chi_{k,n}(\omega,x)\to\chi_{\infty,n}(\omega,x)\text{ as }k\to\infty\text{ for almost all }x\in(0,1)\,.
\end{equation}
To that end, fix $\omega\in A$. If $F_{k,n}$ denotes the c.d.f. of $\esd( W_{N_k,n}(\omega)/\sqrt {N_k})$, then it is easy to see that
\[
 \chi_{k,n}(\omega,x)=F_{k,n}^\leftarrow(x)\,.
\]
By the choice of the set $A$, it follows that for fixed $1\le n<\infty$,
\[
 \lim_{k\to\infty}F_{k,n}(x)=F_n(x)
\]
for all $x$ which is a continuity point of $F_n$. Therefore, by standard analytic arguments (see for example the proof of Theorem 25.6, page 333 in \cite{billingsley:1995}), \eqref{t6.eq3} follows, which in turn establishes \eqref{t6.eq2}. 

The next task is to show that for fixed $1\le n<\infty$, the family 
\[
 \{\chi^2_{k,n}:1\le k<\infty\}\text{ is uniformly integrable}.
\]
To that end it suffices to show that
\[
 \sup_{1\le k<\infty}\mathbb E\left(\chi^4_{k,n}\right)<\infty\,.
\]
Fix $n$ and notice that
\begin{eqnarray*}
 \mathbb E\left(\chi^4_{k,n}\right)&=&N_k^{-3}\E\Tr\left(W_{N_k,n}^4\right)\\
 &\to&\beta_{n,4}\text{ as }k\to\infty\,,
\end{eqnarray*}
the last step following by \eqref{ac.p1.eq1}. This establishes the uniform integrability, which along with \eqref{t6.eq2}, proves that
\begin{equation}\label{t6.eq4}
 \lim_{k\to\infty}\mathbb E\left[\left(\chi_{k,n}-\chi_{\infty,n}\right)^2\right]=0\text{ for all }1\le n<\infty\,.
\end{equation}

Our final claim is that
\begin{equation}\label{t6.eq5}
 \lim_{n\to\infty}\limsup_{k\to\infty}\mathbb E\left[\left(\chi_{k,n}-\chi_{k,\infty}\right)^2\right]=0\,.
\end{equation}
To that end, notice that
\begin{eqnarray*}
 \mathbb E\left[\left(\chi_{k,n}-\chi_{k,\infty}\right)^2\right]&=&N_k^{-2}\E\sum_{j=1}^{N_k}\left[\lambda_j(W_{N_k,n})-\lambda_j(W_{N_k,\infty})\right]^2\\
 &\le&N_k^{-2}\E\Tr\left[\left(W_{N_k,n}-W_{N_k,\infty}\right)^2\right]\\
 &\le&C\sum_{m,l\in\bbz:|m|\vee|l|>n}c_{m,l}^2\,,
\end{eqnarray*}
for some finite constant $C$. The inequality in the second line is the Hoffman-Wielandt inequality; see Lemma 2.1.19 on page 21 in \cite{anderson:guionnet:zeitouni:2010}. This completes the proof of \eqref{t6.eq5}.

Fact \ref{f4} along with \eqref{t6.eq4} and \eqref{t6.eq5} shows that there exists $\chi_{\infty,\infty}\in L^2(\Omega\times(0,1))$ such that
\begin{equation}\label{t6.eq7}
 \lim_{n\to\infty}\mathbb E\left[\left(\chi_{\infty,n}-\chi_{\infty,\infty}\right)^2\right]=0\,.
\end{equation}
It is easy to see that for all $n<\infty$, $\chi_{\infty,n}$ has law $\mu_{f,n}$. Therefore, by \eqref{t1.eq2} and \eqref{t6.eq7}, it follows that law of $\chi_{\infty,\infty}$ is $\mu_f$. Recall that $\beta_{n,2}$, as defined in \eqref{eq.defbeta}, is the second moment of $\mu_{f,n}$. \comment{correction-20} 
Equation  \eqref{t6.eq7} furthermore ensures that
\begin{eqnarray*}
 \int_\bbr x^2\mu_f(dx)
 &=&\lim_{n\to\infty}\int_\bbr x^2\mu_{f,n}(dx)\\
 &=&\lim_{n\to\infty}\beta_{n,2}\\
 &=&\lim_{n\to\infty}2\E(Y_{0,0,n}^2)\\
 &=&2\E(Y_{0,0}^2)\\
 &=&2\int_{[-\pi,\pi]^2} f(x,y)dxdy\,,
\end{eqnarray*}
where $Y_{0,0}$ and $Y_{0,0,n}$ are as in \eqref{eq.defy} and \eqref{eq.defyijn} respectively.
This completes the proof.
\end{proof}

We now proceed towards the proof of Theorem \ref{t2}. For that, we shall need the following two facts, the first of which is a simple consequence of the H\"older inequality. \comment{correction-21,22}

\begin{fact} \label{fact.holder}
Suppose for some integer $k\ge1$ and a measure space $(\Sigma,\Xi,m)$, the functions $\{f_{in}:1\le i\le k,1\le n\le\infty\}$ are in $L^k(\Sigma)$. Furthermore, assume that for all fixed $1\le i\le k$,
\[
 f_{in}\to f_{i\infty},\,\text{as }n\to\infty\text{ in }L^k\,.
\]
Then, the product $f_{1n}\ldots f_{kn}\in L^1(\Sigma)$ for all $1\le n\le\infty$, and
\[
 f_{1n}\ldots f_{kn}\to f_{1\infty}\ldots f_{k\infty},\text{ as }n\to\infty\text{ in }L^1\,.
\]
\end{fact}

The second fact is a restatement of Theorem 3.4.4, page 146 in \cite{krantz:1999}.

\begin{fact}\label{fact.lp}
 Assume that for some $p\in(1,\infty)$, $h\in L^p\left([-\pi,\pi]^2,\bbc\right)$, that is, it is a function from $[-\pi,\pi]^2$ to $\bbc$ with finite $L^p$ norm. Define
\[
 \hat h_{jk}:=\frac1{2\pi}\int_{[-\pi,\pi]^2} e^{-\io(jx+ky)}h(x,y)dxdy,\,j,k\in\bbz\,.
\]
Then,
\[
 \frac1{2\pi}\sum_{j,k=-n}^n\hat h_{j,k}e^{\io(jx+ky)}\to h(x,y)\text{ in the }L^p\text{ norm, as }n\to\infty\,.
\]
\end{fact}

The first step towards proving Theorem \ref{t2} is the following lemma.

\begin{lemma}\label{ac.l5}
 If $f$ is a non-negative trigonometric polynomial defined on $[-\pi,\pi]^2$, that is,
\[
 f(x,y):=\sum_{j,k=-n}^na_{jk}e^{\io(jx+ky)}\ge0\,,
\]
for some finite $n\ge1$, and real numbers $(a_{jk}:1\le j,k\le n)$, then for all fixed $m\ge1$,
\begin{eqnarray}
 &&\int_\bbr x^{2m}\mu_f(dx)\nonumber\\
 &=&\sum_{\sigma\in NC_2(2m)}\sum_{k\in S(\sigma)}\prod_{(u,v)\in\sigma}\int_{[-\pi,\pi]^2}e^{\io(k_ux+k_vy)}[f(x,-y)+\label{ac.l5.claim1}\\
 &&\text{ }\,\text{ }\,\text{ }\,\text{ }\,\text{ }\,\text{ }\,\text{ }\,\text{ }\,\text{ }\,\text{ }\,\text{ }\,f(-y,x)]dxdy\nonumber\\
 &=&(2\pi)^{m-1}\sum_{\sigma\in NC_2(2m)}\int_{[-\pi,\pi]^{m+1}}L_{\sigma,f}(x)dx\,,\label{ac.l5.claim2}
\end{eqnarray}
where $S(\sigma)$ is as in \eqref{eq.defS}.
\end{lemma}

\begin{proof}
 Since $f$ is a trigonometric polynomial, it is integrable, and hence there exists a stationary Gaussian process $(G_{i,j}:i,j\in\bbz)$ with mean zero, and
\begin{equation*}
 R_G(k,l):=\E(G_{0,0}G_{k,l})=\int_{[-\pi,\pi]^2} e^{\io(kx+ly)}f(x,y)dxdy,\,k,l\in\bbz\,.
\end{equation*}
The hypothesis ensures that $R_G(k,l)=0$ if $|k|\vee|l|>n$. Hence, exactly same arguments as those in the proof of Proposition \ref{ac.p1} will show that for fixed $m\ge1$,
\begin{eqnarray*}
 \int_\bbr x^{2m}\mu_f(dx)&=&\sum_{\sigma\in NC_2(2m)}\sum_{k\in S(\sigma)}\prod_{(u,v)\in\sigma}\left[R_G(k_u,-k_v)+R_G(k_v,-k_u)\right]\\
 &=&\sum_{\sigma\in NC_2(2m)}\sum_{k\in S(\sigma)}\prod_{(u,v)\in\sigma}\overline R_G(k_u,k_v)\,,
\end{eqnarray*}
where\comment{correction-24}
\[
 \overline R_G(k,l):=R_G(k,-l)+R_G(l,-k)\,.
\]
Defining
\[
 g(x,y):=f(x,-y)+f(-y,x)\,,
\]
it is easy to see from \eqref{eq.even} using a change of variable that \comment{correction-25} 
\begin{equation}\label{ac.l5.eq3}
 \int_{[-\pi,\pi]^2}e^{\io(kx+ly)}g(x,y)dxdy=\overline R_G(k,l),\,u,v\in\bbz\,,
\end{equation}
which shows \eqref{ac.l5.claim1}.

Therefore, to complete the proof, it suffices to show that
\begin{equation}\label{ac.l5.eq2}
 \sum_{k\in S(\sigma)}\prod_{(u,v)\in\sigma}\overline R_G(k_u,k_v)=(2\pi)^{m-1}\int_{[-\pi,\pi]^{m+1}}L_{\sigma,f}(x)dx\,,
\end{equation}
for all $\sigma\in NC_2(2m)$. To that end, fix $\sigma$, and notice that \eqref{ac.l5.eq3} implies that
\[
 g(x,y)=(2\pi)^{-2}\sum_{k,l=-n}^n\overline R_G(k,l)e^{-\io(kx+ly)}\,,
\]
for almost all $x,y$. Observe that for all $k\in\{-n,\ldots,n\}^{2m}$ and\\ $x\in[-\pi,\pi]^{m+1}$,
\[
 \sum_{(u,v)\in\sigma}\left[k_ux_{\mathcal T_\sigma(u)}+k_vx_{\mathcal T_\sigma(v)}\right]=\sum_{l=1}^{m+1}x_l\sum_{j\in V_l}k_j\,,
\]
and hence \comment{correction-26, 27}
\begin{eqnarray*}
 &&\int_{[-\pi,\pi]^{m+1}}L_{\sigma,f}(x)dx\\
 &=&\int_{[-\pi,\pi]^{m+1}}\left[\prod_{(u,v)\in\sigma}g\left(x_{\mathcal T_\sigma(u)},x_{\mathcal T_\sigma(v)}\right)\right]dx\\
 &=&(2\pi)^{-2m}\int_{[-\pi,\pi]^{m+1}}\Biggl[\sum_{k\in\{-n,\ldots,n\}^{2m}}\exp\left(\iota\sum_{l=1}^{m+1}x_l\sum_{j\in V_l}k_j\right)\\ 
 &&\,\,\,\,\,\,\,\,\,\,\,\,\,\,\,\,\,\,\prod_{(u,v)\in\sigma}\overline R_G(k_u,k_v)\Biggr]dx\\
 &=&(2\pi)^{1-m}\sum_{k\in\{-n,\ldots,n\}^{2m}\cap S(\sigma)}\prod_{(u,v)\in\sigma}\overline R_G(k_u,k_v)\\
 &=&(2\pi)^{1-m}\sum_{k\in S(\sigma)}\prod_{(u,v)\in\sigma}\overline R_G(k_u,k_v)\,.
\end{eqnarray*}
This shows \eqref{ac.l5.eq2} which in turn establishes \eqref{ac.l5.claim2} and hence completes the proof.
\end{proof}

The following lemma will also be needed for the proof of Theorem \ref{t2}.

\begin{lemma}\label{ac.l3}
 Suppose that for all $1\le n\le\infty$, $g_n$ is a non-negative, integrable and even function on $[-\pi,\pi]^2$ such that as $n\to\infty$,
\[
{g_n} \to{g_\infty}\text{ in }L^1\,.
\]
Then,
\[
 \mu_{g_n}\weak\mu_{g_\infty}\text{ as }n\to\infty\,.
\]
\end{lemma}

\begin{proof}
The hypothesis can be restated as
\begin{equation}\label{ac.l3.eq2}
 \sqrt{g_n} \to\sqrt{g_\infty}\text{ in }L^2\,.
\end{equation}

Let $(G_{i,j}:i,j\in\bbz)$ be a family of i.i.d. $N(0,1)$ random variables.
 Define
\[
 d_{k,l,n}:=(2\pi)^{-1}\int_{[-\pi,\pi]^2} e^{-\io(kx+ly)}\sqrt{g_n(x,y)}dxdy,\,k,l\in\bbz,1\le n\le\infty\,,
\]
and
\[
 H_{i,j,n}:=\sum_{k,l\in\bbz}d_{k,l,n}G_{i-k,j-l},\,i,j\in\bbz,1\le n\le\infty\,.
\]
By Fact \ref{fact:linprocess}, it follows that for all $1\le n\le\infty$, the family $(H_{i,j,n}:i,j\in\bbz)$ is a stationary Gaussian process whose spectral density is $g_n$. For every $1\le n\le\infty$ and $1\le N<\infty$, define an $N\times N$ matrix $A_{N,n}$ by
\[
 A_{N,n}(i,j):=\left(H_{i,j,n}+H_{j,i,n}\right)/\sqrt N,\,1\le i,j\le N\,.
\]
By Theorem \ref{t1}, it follows that for all $1\le n\le\infty$,
\begin{equation}\label{ac.l3.eq1}
 L\left(\esd(A_{N,n}),\mu_{g_n}\right)\prob0\text{ as }N\to\infty\,.
\end{equation}
Notice that for fixed $1\le N,n<\infty$, by arguments similar to those leading to \eqref{ac.t1.eq2} from \eqref{ac.t1.eq1},
\begin{eqnarray*}
 \E\Tr\left[(A_{N,n}-A_{N,\infty})^2/N\right]&\le&4\sum_{k,l\in\bbz}\left(d_{k,l,n}-d_{k,l,\infty}\right)^2\\
 &=&4\int_{[-\pi,\pi]^2}\left(\sqrt{g_n(x,y)}-\sqrt{g_\infty(x,y)}\right)^2dxdy\,,
\end{eqnarray*}
the last equality following from Parseval. Therefore, by \eqref{ac.l3.eq2}, it holds that for all $\vep>0$,
\[
 \lim_{n\to\infty}\limsup_{N\to\infty}P\left[L\left(\esd(A_{N,n}),\esd(A_{N,\infty})\right)>\vep\right]=0\,.
\]
The above, along with \eqref{ac.l3.eq1} and Fact \ref{f4} proves the claim of the lemma.
\end{proof}

\begin{proof}[Proof of Theorem \ref{t2}]
Fix $m\ge2$, and assume that $\|f\|_m<\infty$. Let $c_{kl}$ be as in \eqref{eq.defc}, and define for $n\ge1$,
\begin{equation}\label{eq.deffn}
 f_n(x,y):=\left[\frac1{2\pi}\sum_{k,l=-n}^nc_{kl}e^{\io(kx+ly)}\right]^2,\,-\pi\le x,y\le\pi\,.
\end{equation}
By Fact \ref{fact.lp}, it follows that
\begin{equation}\label{t2.eq1}
 f_n\to f\text{ in }L^m\text{ norm, as }n\to\infty\,.
\end{equation}
Fix $\sigma\in NC_2(2m)$.
Equation \eqref{t2.eq1} along with the observation that for all $(u,v)\in\sigma$, ${\mathcal T}_\sigma(u)\neq{\mathcal T}_\sigma(v)$ implies that\newcomment{Comment 12}
\[
\lim_{n\to\infty}\int_{\bbr^{m+1}}\left|f_n\left(x_{{\mathcal T}_\sigma(u)},-x_{{\mathcal T}_\sigma(v)}\right)-f\left(x_{{\mathcal T}_\sigma(u)},-x_{{\mathcal T}_\sigma(v)}\right)\right|^mdx_1\ldots dx_{m+1}=0\,.
\]
This with an appeal to Fact \ref{fact.holder}, implies that
\begin{equation}\label{t2.eq2}
 \lim_{n\to\infty}\int_{[-\pi,\pi]^{m+1}}L_{\sigma,f_n}(x)dx=\int_{[-\pi,\pi]^{m+1}}L_{\sigma,f}(x)dx\,.
\end{equation}
Equation \eqref{t2.eq1} along with Lemma \ref{ac.l3} and the observation that $f_n$ is a non-negative even function implies that
\begin{equation}\label{t2.eq3}
 \mu_{f_n}\weak\mu_f\text{ as }n\to\infty\,.
\end{equation}
Therefore, by Fatou's lemma, it follows that
\begin{eqnarray}
\nonumber &&\int_\bbr x^{2m}\mu_f(dx)\\
\label{t2.eq4} &\le&\liminf_{n\to\infty}\int_\bbr x^{2m}\mu_{f_n}(dx)\\
\nonumber &=&\liminf_{n\to\infty}(2\pi)^{m-1}\sum_{\sigma\in NC_2(2m)}\int_{[-\pi,\pi]^{m+1}}L_{\sigma,f_n}(x)dx\\
\label{t2.eq5} &=&(2\pi)^{m-1}\sum_{\sigma\in NC_2(2m)}\int_{[-\pi,\pi]^{m+1}}L_{\sigma,f}(x)dx<\infty\,,
\end{eqnarray}
the equality in the last two  lines following from Lemma \ref{ac.l5} and \eqref{t2.eq2} respectively. This completes the proof.
\end{proof}

\begin{proof}[Proof of Theorem \ref{theorem.new}]
Assume that $\|f\|_\infty<\infty$. In the proof of Theorem~\ref{t2} above,\comment{correction-29}
 it has essentially been shown that the limit in \eqref{t2.eq4} exists, and equals the quantity in \eqref{t2.eq5}, that is,
\[
 \lim_{n\to\infty}\int_\bbr x^{2m}\mu_{f_n}(dx)=
\]
\[
(2\pi)^{m-1}\sum_{\sigma\in NC_2(2m)}\int_{[-\pi,\pi]^{m+1}}L_{\sigma,f}(x)dx\text{ for all }m\ge1\,,
\]
where $f_n$ is as in \eqref{eq.deffn},
and
\begin{eqnarray*}
 &&\limsup_{m\to\infty}\left[(2\pi)^{m-1}\sum_{\sigma\in NC_2(2m)}\int_{[-\pi,\pi]^{m+1}}L_{\sigma,f}(x)dx\right]^{1/2m}\\
 &\le&\limsup_{m\to\infty}\left[(2\pi)^{2m}(2\|f\|_\infty)^{m+1}\#NC_2(2m)\right]^{1/2m}\\
 &=&\bar R\,.
\end{eqnarray*}
This shows that there exists an even probability measure $\mu^*$ supported on $[-\bar R,\bar R]$ such that
\[
 \int_\bbr x^{2m}\mu^*(dx)=\sum_{\sigma\in NC_2(2m)}\int_{[-\pi,\pi]^{m+1}}L_{\sigma,f}(x)dx\text{ for all }m\ge1\,,
\]
and
\[
 \mu_{f_n}\weak\mu^*\text{ as }n\to\infty\,.
\]
This, along with \eqref{t2.eq3} completes the proof of both parts.
\end{proof}

For the proof of Theorem \ref{thm.stieltjes}, we shall need the following result.

\begin{lemma}\label{stieltjes.l1}
Define for all $m\ge1$ and $\sigma\in NC_2(2m)$,
\[
 h_\sigma(y):=\int_{-\pi}^\pi\ldots\int_{-\pi}^\pi L_{\sigma,f}(x_1,\ldots,x_m,y)dx_m\ldots dx_1,\,y\in[-\pi,\pi]\,.
\]
 Assume that $\sigma\in NC_2(2m)$ can be written as 
\begin{equation}\label{stieltjes.l1.eq1}
 \sigma=\sigma_1\cup\sigma_2\,,
\end{equation}
where $\sigma_1\in NC_2(2k)$ for some $1\le k\le m-1$, and $\sigma_2$ is a non-crossing pair partition of $\{2k+1,\ldots,2m\}$. Viewing $\sigma_2$ as an element of $NC_2(2m-2k)$ by the obvious relabeling of $2k+1,\ldots,2m$ to $1,\ldots,2m-2k$ respectively, it is true that
\[
 h_\sigma(y)=h_{\sigma_1}(y)h_{\sigma_2}(y),\,y\in[-\pi,\pi]\,.
\]
\end{lemma}

\begin{proof}It is easy to see from \eqref{stieltjes.l1.eq1} and the fact that ${\mathcal T}(2m)=m+1$, that
\begin{equation}\label{stieltjes.l1.eq2}
 \mathcal T_\sigma(j)\in\{1,\ldots,k,m+1\},\mbox{ for }1\le j\le2k\,,
\end{equation}
and
\begin{equation}\label{stieltjes.l1.eq3}
 \mathcal T_\sigma(j)\in\{k+1,\ldots,m+1\},\mbox{ for }2k+1\le j\le2m\,.
\end{equation}
Define
\begin{equation}\label{stieltjes.l1.eq4}
 g(x,y):=f(x,-y)+f(-y,x),\,x,y\in[-\pi,\pi]\,.
\end{equation}
Therefore,
\begin{eqnarray*}
&&L_{\sigma,f}(x)\\
&=&\prod_{(u,v)\in\sigma}g\left(x_{\mathcal T_\sigma(u)},x_{\mathcal T_\sigma(v)}\right)\\
&=&\left(\prod_{(u,v)\in\sigma:u,v\le2k}g\left(x_{\mathcal T_\sigma(u)},x_{\mathcal T_\sigma(v)}\right)\right) \left(\prod_{(u,v)\in\sigma:u,v>2k}g\left(x_{\mathcal T_\sigma(u)},x_{\mathcal T_\sigma(v)}\right)\right)\,.
\end{eqnarray*}
By \eqref{stieltjes.l1.eq2} and \eqref{stieltjes.l1.eq3}, it follows that
\begin{eqnarray*}
&& h_\sigma(x_{m+1})\\
&=&\left(\int_{-\pi}^\pi \ldots\int_{-\pi}^\pi \prod_{(u,v)\in\sigma:u,v\le2k}g\left(x_{\mathcal T_\sigma(u)},x_{\mathcal T_\sigma(v)}\right)dx_k\ldots dx_1\right)\\
&&\left(\int_{-\pi}^\pi \ldots\int_{-\pi}^\pi \prod_{(u,v)\in\sigma:u,v>2k}g\left(x_{\mathcal T_\sigma(u)},x_{\mathcal T_\sigma(v)}\right)dx_m\ldots dx_{k+1}\right)\\
&=&h_{\sigma_1}(x_{m+1})h_{\sigma_2}(x_{m+1})\,.
\end{eqnarray*}
This completes the proof.
\end{proof}

\begin{lemma}\label{stieltjes.l2}
 If
\begin{equation*}
 \sigma=\{(1,2m)\}\cup\sigma_1\,,
\end{equation*}
for some non-crossing pair partition $\sigma_1$ of $\{2,\ldots,2m-1\}$, then
\[
 h_\sigma(z)=\int_{-\pi}^\pi h_{\sigma_1}(y)g(y,z)dy,\,z\in\bbr\,,
\]
where, once again, $\sigma_1$ is viewed as an element of $NC_2(2m-2)$, and $g$ is as in \eqref{stieltjes.l1.eq4}.
\end{lemma}

\begin{proof}
 Throughout the proof, $\sigma_1$ is to be thought of as an element of \\$NC_2(2m-2)$. Clearly,
\begin{eqnarray*}
\mathcal T_\sigma(1)&=&m\,,\\
\mathcal T_\sigma(2m)&=&m+1\,,\\
\mathcal T_\sigma(j)&=&\mathcal T_{\sigma_1}(j-1),\,2\le j\le2m-1\,.
\end{eqnarray*}
The above equations imply that
\[
 L_{\sigma,f}(x)=g(x_m,x_{m+1})L_{\sigma_1}(x_1,\ldots,x_m),\,x\in\bbr^{m+1}\,.
\]
Thus,
\begin{eqnarray*}
 h_\sigma(z)&=&\int_{-\pi}^\pi \ldots \int_{-\pi}^\pi g(x_m,z)L_{\sigma_1}(x_1,\ldots,x_m)dx_m\ldots dx_1\\
&=&\int_{-\pi}^\pi g(x_m,z)\left\{\int_{-\pi}^\pi\ldots\int_{-\pi}^\pi L_{\sigma_1}(x_1,\ldots,x_m)dx_{m-1}\ldots dx_1\right\}dx_m\\
&=&\int_{-\pi}^\pi g(x_m,z)h_{\sigma_1}(x_m)dx_m\,,
\end{eqnarray*}
which completes the proof.
\end{proof}

\begin{proof}[Proof of Theorem \ref{thm.stieltjes}]
Define $H_0(x):=1$ for all $-\pi\le x\le\pi$, and for $m\ge1$,
\[
 H_{2m}(x):=(2\pi)^{m-1}\sum_{\sigma\in NC_2(2m)}h_\sigma(x),\,x\in[-\pi,\pi]\,.
\]
Clearly,
\begin{eqnarray*}
 &&\sup_{-\pi\le x\le\pi}\limsup_{m\to\infty}H_{2m}(x)^{1/2m}\\
 &\le&\lim_{m\to\infty}\left[(2\pi)^{m-1}(4\pi\|f\|_\infty)^m\#NC_2(2m)\right]^{1/2m}\\
 &=&\bar R\,.
\end{eqnarray*}
Therefore, for all $x\in[-\pi,\pi]$, the power series
\[
 {\mathcal H}(z,x):=\sum_{m=0}^\infty\frac{H_{2m}(x)}{z^{2m+1}}
\]
converges in $\{z\in\bbc:|z|>\bar R\}$. Clearly, $\mathcal H$ satisfies claims (1), (2) and (4) of the result. To show claim (3), we derive a recursion for $\mathcal H(z,x)$ using the
  properties of $h_\sigma(x)$.
A Catalan word of length $2m$ has $m$ letters which appear twice and  successive deletion of double letters leads to the empty word. For example the words $abbcca$ and $abccbadd$ are Catalan words while $abab$ and $abccab$ are not. The reader is referred to \cite{bose:sen:2008} for more details on Catalan words.
\comment{correction-30}\newcomment{Comment 13}
Recall that there is a natural bijection between $NC_2(2m)$ and the set of Catalan words of length $2m$ with the understanding that two words will be considered identical if one can be obtained from the other by a relabeling of letters. Keeping this correspondence in mind, by an abuse of notation, we shall now consider $h_w(x)$ for Catalan words $w$, and denote by $NC_2(2m)$ the set of Catalan words of length $2m$.
 Note that any Catalan word $w$
  of length $2m $ can be written as $w=aw_1aw_2$, for some $w_1\in NC_2(2k-2)$
  and $w_2\in NC_2(2m-2k)$.  So if
  $$H_{2m,k}(x):=\sum_{w_1\in NC_2(2k-2)}\sum_{w_2\in NC_2(2m-2k)}h_{aw_1aw_2}(x)\,,$$ then
  $$H_{2m}(x)=(2\pi)^{m-1}\sum_{k=1}^mH_{2m,k}(x).$$
Notice that
  \begin{align*}
    H_{2m,k}(x)
    &=\sum_{w_1\in NC_2(2k-2)}h_{aw_1a}(x)\sum_{w_2\in NC_2(2m-2k)}h_{w_2}(x)\\
    &=\sum_{w_1\in NC_2(2k-2)}\int_{-\pi}^\pi\left[g(x,y)h_{w_1}(y)\right]dy\sum_{w_2\in NC_2(2m-2k)}h_{w_2}(x)\\
    &=(2\pi)^{3-m}\int_{-\pi}^\pi\left[g(x,y)H_{2(k-1)}(y) H_{2(m-k)}(x)\right]dy\,,
  \end{align*}
the equalities in the first two lines following from Lemmas \ref{stieltjes.l1} and \ref{stieltjes.l2} respectively.
  As a consequence,
\begin{equation}\label{thm.stieltjes.eq1}
 H_{2m}(x)=4\pi^2\sum_{k=1}^mH_{2(m-k)}(x) \int_{-\pi}^\pi g(x,y)H_{2(k-1)}(y)dy\,.
\end{equation}
Now by an easy computation it follows that,
$$\mathcal H(z,x)=\sum_{m=0}^{\infty}H_{2m}(x)z^{-(2m+1)}=\frac1z+\frac{4\pi^2}z\mathcal{H}(z,x)\int_{-\pi}^\pi g(x,y)\mathcal{H}(z,y)dy.$$
This shows \eqref{thm.stieltjes.claim1}, that is, claim (3).

Next, we proceed to show uniqueness of the function satisfying (1) - (4). To that end, let $\mathcal{\tilde H}$ be another solution. By (1), it holds that for all fixed $x$, ${\mathcal{\tilde H}}(\cdot,x)$ has a Laurent series expansion on $\{z\in\bbc:|z|>\bar R\}$. By (4), it follows that
\[
 {\mathcal{\tilde H}}(z,x)=\sum_{m=0}^\infty \tilde H_m(x)z^{-m-1},\,(z,x)\in\mathcal D\,,
\]
for some $\tilde H_0(x),\tilde H_1(x),\ldots\in\bbc$, with $\tilde H_0(x)\equiv1$. Condition (3) is equivalent to
\[
 \tilde H_m(x)=\sum_{k=0}^{m-2}4\pi^2\tilde H_k(x)\int_{-\pi}^\pi\tilde H_{m-2-k}(y)g(x,y)dy,\,m\ge1\,.
\]
From here, inductively it can be argued that $\tilde H_m(x)\equiv0$ for all odd $m$, and that \eqref{thm.stieltjes.eq1} holds with $H$ replaced by $\tilde H$. This establishes the uniqueness.

Finally, \eqref{thm.stieltjes.claim2} follows the second claim of Theorem \ref{theorem.new}.
\end{proof}

Next, we proceed towards the proof of Theorem \ref{t3}. The following lemma, which is the first step towards that, proves the result for a special case.

\begin{lemma}\label{ac.l2}
Suppose that $(G_{i,j}:i,j\in\bbz)$ is a stationary Gaussian process whose covariance kernel $R_G(\cdot,\cdot)$ defined by
\[
 R_G(u,v):=\E\left[G_{0,0}G_{u,v}\right],\,u,v\in\bbr\,,
\]
satisfies
\[
 R_G(u,v)=\int_{[-\pi,\pi]^2} e^{\io(ux+vy)}f_G(x)f_G(y)\,dxdy,\,u,v\in\bbr\,,
\]
for some non-negative $f_G(\cdot)$ defined on $[-\pi,\pi]$, and there exists $n$ such that
\begin{equation}\label{ac.l2.eq1}
 R_G(u,v)=0\text{ if }|u|\vee|v|>n\,.
\end{equation}
Then $\esd$ of the $N\times N$ matrix whose $(i,j)$-th entry is $G_{i,j}/\sqrt N$ converges weakly in probability to
\[
 \eta_G\boxtimes WSL(1)\,,
\]
where $\eta_G$ is the law of $f_G(U)\pi2\sqrt2$ and $U$ is a $Uniform(-\pi,\pi)$ random variable.
\end{lemma}

\begin{proof}
 By Theorem \ref{t1}, it follows that the limit exists, say $\mu_G$, and furthermore by the hypothesis \eqref{ac.l2.eq1}, and claim \eqref{ac.l5.claim1} of Lemma \ref{ac.l5}, it holds that
\begin{equation}\label{ac.l2.eq2}
 \int_{\mathbb R} x^{2m}\mu_G(dx)=
\end{equation}
\[
 \sum_{\sigma\in NC_2(2m)}\sum_{k\in S(\sigma)}\prod_{(u,v)\in\sigma}\left[R_G(k_u,-k_v)+R_G(k_v,-k_u)\right],\,m\ge1\,.
\]
Our first claim is that
\begin{equation}\label{ac.l2.eq3}
 R_G(u,-v)+R_G(v,-u)=r_G(u)r_G(v),\,u,v\in\bbz\,,
\end{equation}
where 
\begin{equation}\label{ac.l2.eq4}
 r_G(u):=\sqrt2\int_{-\pi}^\pi e^{\io ux}f_G(x)dx\,.
\end{equation}
To that end, notice that by \eqref{eq.even}, it follows that
\[
 f_G(-x)f_G(-y)=f_G(x)f_G(y)\text{ for almost all }x,y\in[-\pi,\pi]\,.
\]
Integrating out $y$, it follows that $f_G(x)=f_G(-x)$ for almost all $x\in [-\pi,\pi]$.\comment{correction-31} Therefore,
\begin{eqnarray*}
 R_G(u,-v)+R_G(v,-u)&=&\int_{[-\pi,\pi]^2}2\cos(ux-vy)f_G(x)f_G(y)dxdy\\
 &=&r_G(u)r_G(v)\,,
\end{eqnarray*}
where the fact that $f_G(\cdot)$ is even almost surely has been used for the last equality. This establishes \eqref{ac.l2.eq3}.

Our next claim is that
\begin{equation}\label{ac.l2.eq5}
f_G(x)=\frac1{\pi2\sqrt2}\sum_{k=-n}^n r_G(k)e^{-\io kx},\text{ for almost all }x\in[-\pi,\pi]\,.
\end{equation}
The above follows from \eqref{ac.l2.eq4} using Fourier inversion and the fact that for $|u|>n$, $r_G(u)=0$ which is a consequence of \eqref{ac.l2.eq1}.

By \eqref{ac.l2.eq2} and \eqref{ac.l2.eq3}, it follows that for fixed $m\ge1$,
\begin{eqnarray*}
 \int x^{2m}\mu_G(dx)&=&\sum_{\sigma\in NC_2(2m)}\sum_{k\in S(\sigma)}\prod_{(u,v)\in\sigma}r_G(k_u)r_G(k_v)\\
 &=&\sum_{\sigma\in NC_2(2m)}\sum_{k\in S(\sigma)}\prod_{j=1}^{2m}r_G(k_j)\,.
\end{eqnarray*}
Fix $\sigma\in NC_2(2m)$, and let for $i=1,\ldots,m+1$, $l_i$ be the size of $V_i$ which are the blocks of the Kreweras complement of $\sigma$, as in \eqref{eq.defV}. Then, it is easy to see that
\begin{eqnarray*}
\sum_{k\in S(\sigma)}\prod_{j=1}^{2m}r_G(k_j)&=&\prod_{i=1}^{m+1}\sum_{k\in\bbz^{l_i}:k_1+\ldots+k_{l_i}=0}\prod_{j=1}^{l_i}r_G(k_j)\\
&=&\prod_{i=1}^{m+1}(2\pi)^{-1}\int_{-\pi}^\pi \left[f_G(x)\pi2\sqrt2\right]^{l_i}dx\\
&=&\prod_{i=1}^{m+1}\int_\bbr x^{l_i}\eta_G(dx)\,,
\end{eqnarray*}
the second last equality being a consequence of \eqref{ac.l2.eq5}.
Therefore, it follows that
\begin{eqnarray*}
 \int_\bbr x^{2m}\mu_G(dx)&=&\sum_{\sigma\in NC_2(2m)}\prod_{i=1}^{m+1}\int_\bbr x^{l_i}\eta_G(dx)\,.
\end{eqnarray*}
Equation (14.5) on page 228 in \cite{nica:speicher:2006}, which is a consequence of Theorem 14.4 therein, implies that the right hand side of above equals
\comment{correction-32}
\[
\int_\bbr x^{2m}\eta_G\boxtimes WSL(1)\,.
\]
 This shows that $\mu_G=\eta_G\boxtimes WSL(1)$, and thus completes the proof. 
\end{proof}

The next step towards the proof of Theorem \ref{t3} is the following. 

\begin{lemma}\label{ac.l4}
If
\[
 g(x,y):=\frac12\left[f(x,y)+f(y,x)\right],\,-\pi\le x,y\le\pi\,,
\]
then
\[
 \mu_f=\mu_g\,.
\]
\end{lemma}

\begin{proof}
 For $n\ge1$, let $f_n$ be as in \eqref{eq.deffn}, and define
and
\[
 g_n(x,y):=\frac12\left[f_n(x,y)+f_n(y,x)\right],\,-\pi\le x,y\le\pi\,.
\]
Noticing that for all $\sigma\in\bigcup_{m=1}^\infty NC_2(2m)$, 
\[
 L_{\sigma,f_n}=L_{\sigma,g_n}\text{ a.e.}\,,
\]
it follows by Lemma \ref{ac.l5} that
\[
 \mu_{f_n}=\mu_{g_n}\text{ for all }n\ge1\,.
\]
Using \eqref{t2.eq1} with $m=1$, which is valid because $\|f\|_1<\infty$, it follows that
\[
 {f_n}\to f\text{ in }L^1\text{ as }n\to\infty\,,
\]
from which it follows that
\[
 g_n\to g\text{ in }L^1\,.
\]
Lemma \ref{ac.l3} completes the proof.
\end{proof}

\begin{proof}[Proof of Theorem \ref{t3}]
Define 
\[
 g(x,y):=\frac12\left[f(x,y)+f(y,x)\right],\,-\pi\le x,y\le\pi\,.
\]
In view of Lemma \ref{ac.l4}, it suffices to show that
\begin{equation}\label{t3.eq1}
 \mu_g=\eta_r\boxtimes WSL(1)\,.
\end{equation}
To that end, define
\[
 d_k:=(2\pi)^{-1/2}\int_{-\pi}^\pi e^{-\io kx}\sqrt{r(x)}dx,\,k\in\bbz\,.
\]
Then, it is easy to see that
\begin{equation}\label{t3.eq2}
 (2\pi)^{-1}\int_{[-\pi,\pi]^2} e^{-\io(kx+ly)}\sqrt{g(x,y)}dxdy=d_kd_l,\,k,l\in\bbz\,.
\end{equation}
Define
\[
 g_n(x,y):=\left[(2\pi)^{-1}\sum_{k,l=-n}^nd_kd_le^{\io(kx+ly)}\right]^2,\,-\pi\le x,y\le\pi\,.
\]
Clearly,
\[
 g_n(x,y)=r_n(x)r_n(y)\text{ for all }-\pi\le x,y\le\pi\,,
\]
where
\[
 r_n(x):=\left[(2\pi)^{-1/2}\sum_{k=-n}^nd_ke^{\io kx}\right]^2\,.
\]
Arguments similar to those in the proof of Lemma \ref{ac.l2} show that $g_n(\cdot,\cdot)$ and $r_n(\cdot)$ take values in the non-negative half line.
By the same lemma, it follows that
\[
 \mu_{g_n}=\eta_{r_n}\boxtimes WSL(1),\,n\ge1\,,
\]
where $\eta_{r_n}$ is the law of $2^{3/2}\pi r_n(U)$, $U$ being a  $Uniform(-\pi,\pi)$ random variable. By the Fourier inversion theorem, it follows that
\[
 \lim_{n\to\infty}\int_{-\pi}^\pi\left|r_n(x)-r(x)\right|dx=0\,,
\]
and hence 
\[
\eta_{r_n}\weak\eta_r\text{ as }n\to\infty\,.
\]
\comment{correction-33}
For a probability measure $Q$ on $\bbr$, let $Q^2$ denote its push forward under the map $x\mapsto x^2$. Thus, $WSL^2(1)$, the push forward of $WSL(1)$, is a probability measure on $[0,\infty)$. Corollary 6.7 of \cite{bercovici:voiculescu:1993} implies that
\[
\eta_{r_n}\boxtimes WSL^2(1)\boxtimes\eta_{r_n}\weak\eta_r\boxtimes WSL^2(1)\boxtimes\eta_r\text{ as }n\to\infty\,.
\]
By  Lemma 8 of \cite{arizmendi:abreu:2009}, it follows that the left hand side and right hand side above are $(\eta_{r_n}\boxtimes WSL(1))^2$ and  $(\eta_{r}\boxtimes WSL(1))^2$ respectively. Hence,
\[
 \eta_{r_n}\boxtimes WSL(1)\weak\eta_r\boxtimes WSL(1)\text{ as }n\to\infty\,,
\]
which follows because the probability measures above are symmetric about zero.
Again, Fourier inversion and \eqref{t3.eq2} tells us that
\[
 \lim_{n\to\infty}\int_{[-\pi,\pi]^2}\left|g_n(x,y)-g(x,y)\right|dxdy=0\,.
\]
An appeal to Lemma \ref{ac.l3} establishes \eqref{t3.eq1}, and thus completes the proof.
\end{proof}

Finally, we prove Theorem \ref{t4}. 

\begin{proof}[Proof of Theorem \ref{t4}]
 Define
\[
 g(x,y):=\frac12\left[f(x,y)+f(y,x)\right],\,-\pi\le x,y\le\pi\,.
\]
By Lemma \ref{ac.l4}, it suffices to show that
\begin{equation}\label{t4.eq1}
 \mu_g=WSL(2\|f\|_1)\,.
\end{equation}
To that end, set
\comment{correction-34}
\[
d_{j,k}= (2\pi)^{-1}\int_{[-\pi,\pi]^2} e^{-\iota (jx+ky)}\sqrt{g(x,y)}dx dy
\]
and
\[
 h_n(x,y):=(2\pi)^{-1}\sum_{j,k\in\bbz}d_{j,k}\one(j,k\in A_n)e^{\io(jx+ky)},\,-\pi\le x,y\le\pi\,.
\]
Since $d_{j,k}$ are the Fourier coefficients of $\sqrt g$, which is even by~\eqref{eq.even}, they are real numbers, and furthermore by the Parseval's identity, it follows that
\[
 \sum_{j,k\in\bbz}d_{j,k}^2<\infty\,,
\]
and hence in view of the assumption that $A_n\uparrow\bbz$, it follows that as $n\to\infty$,
\[
 h_n\to\sqrt g\text{ in }L^2\,.
\]
Define
\[
 g_n(\cdot,\cdot):=|h_n(\cdot,\cdot)|^2,\,n\ge1\,,
\]
where the modulus is necessary because $h_n(\cdot)$ is $\bbc$-valued. Therefore,
\begin{equation}\label{t4.eq2}
 g_n\to g\text{ in }L^1\,.
\end{equation}

Fix $n\ge1$. Since $d_{j,k}$ is real, it is easy to see that
\begin{eqnarray*}
 &&g_n(x,y)\\
 &=&(2\pi)^{-2}\sum_{i,j,k,l\in\bbz}d_{i,j}d_{k,l}\one(i,j,k,l\in A_n)e^{\io((i-k)x+(j-l)y)}\\
 &=&(2\pi)^{-2}\sum_{u,v\in\bbz}e^{-\io(ux+vy)}\sum_{i,j\in\bbz}d_{i,j}d_{i+u,j+v}\one(i,j,i+u,j+v\in A_n)\,.
\end{eqnarray*}

Since $A_n$ is a finite set, $g_n$ is a trigonometric polynomial.  By \eqref{ac.l5.claim1} of Lemma \ref{ac.l5} and the observation that $g_n(x,y)=g_n(y,x)$, it follows that for all $m\ge1$,
\begin{equation}\label{t4.eq3}
\int_\bbr x^{2m}\mu_{g_n}(dx)=
\end{equation}
\[
 \sum_{\sigma\in NC_2(2m)}\sum_{k\in S(\sigma)} \prod_{(u,v)\in\sigma}2\int_{[-\pi,\pi]^2} e^{\io(k_ux+k_vy)}g_n(x,-y)dxdy\,.
\]
Fix $u\in\bbz\setminus\{0\}$, and notice that
\begin{eqnarray}
\label{t4.eq5} &&\int_{[-\pi,\pi]^2} e^{\io ux}g_n(x,-y)dxdy\\
\nonumber & =&\sum_{i,j\in\bbz}d_{i,j}d_{i+u,j}\one(i,j,i+u\in A_n)=0\,.
\end{eqnarray}
From the above, a simple induction on $m$ will show that for all $\sigma\in NC_2(2m)$ and for all $k\in S(\sigma)$,
\[
 \prod_{(u,v)\in\sigma}\int_{[-\pi,\pi]^2} e^{\io(k_ux+k_vy)}g_n(x,-y)dxdy\neq0
\]
implies that $k=(0,\ldots,0)$. A proof of the induction step is sketched in the following lines. Assume that for a fixed $m\ge1$, the claim is true for all $\sigma\in NC_2(2m)$ and for all $k\in S(\sigma)$. Now, fix a $\pi\in NC_2(2m+1)$, and assume that for some $k\in S(\pi)$,
\begin{equation}\label{t4.eq4}
 \prod_{(u,v)\in\pi}\int_{[-\pi,\pi]^2} e^{\io(k_ux+k_vy)}g_n(x,-y)dxdy\neq0\,.
\end{equation}
By definition of non-crossing pair partitions, there exist $1\le u\le2m+1$ such that $(u,u+1)\in\pi$, and hence $(\ol u)$ is a block of the Kreweras complement of $\pi$; the reader may refer to page \pageref{pg.combinatorics} for the definition of the Kreweras complement. Therefore, $k_u=0$. Equation \eqref{t4.eq4} implies that
\[
\int_{[-\pi,\pi]^2} e^{\io(k_ux+k_{u+1}y)}g_n(x,-y)dxdy\neq0\,,
\]
and hence it follows from \eqref{t4.eq5} that $k_{u+1}=0$. Let $\sigma$ denote the element of $NC_2(2m)$ obtained from $\pi$ by deleting $(u,u+1)$, and relabeling the indices in the natural way. That is, if for all $1\le j\le2m$,
\[
\gamma(j):=
\begin{cases}
j,&1\le j\le u-1\,,\\
j+2,&u\le j\le2m\,,
\end{cases}
\]
then $\sigma=\{(\gamma(u),\gamma(v)):(u,v)\in\pi\}$. It is then easy to see that $(k_{\gamma(u)}:1\le u\le2m)\in S(\sigma)$ because $k_{u+1}=0$, and furthermore \eqref{t4.eq4} implies that
\[
 \prod_{(u,v)\in\sigma}\int_{[-\pi,\pi]^2} e^{\io(k_{\gamma(u)}x+k_{\gamma(v)}y)}g_n(x,-y)dxdy\neq0\,.
\]
Thus, the induction hypothesis applies, implying that $k_{\gamma(u)}=0$ for all $1\le u\le2m$, thereby completing the induction step.
\comment{correction-35}

Therefore, \eqref{t4.eq3} boils down to
\[
 \int_\bbr x^{2m}\mu_{g_n}(dx)=(2\|g_n\|_1)^m\#NC_2(2m),\,m\ge1\,,
\]
and hence
\[
 \mu_{g_n}=WSL(2\|g_n\|_1)\,.
\]
Equation \eqref{t4.eq2} with an appeal to Lemma \ref{ac.l3} shows \eqref{t4.eq1} and thus completes the proof.
\end{proof}

\subsection{Proofs of Theorems \ref{discrete.t1} - \ref{discrete.t2}}
As the first step towards proving Theorem \ref{discrete.t1}, we start with a special case. 

\begin{prop}\label{p1} 
There exists a random point measure $\xi$ which is almost surely in ${\mathcal C}_2$ such that
\begin{equation}\label{p1.eq4}
 d_2\left(\eim(\widetilde W_N/N),\xi\right)\prob0\,,
\end{equation}
as $N\to\infty$, where $\widetilde W_N$ is as in \eqref{eq.deftildew}.
\end{prop}

\begin{remark}\label{rem.r1}
 In view of the inequality
\[
 P\left(d_4(\xi_1,\xi_2)>\vep\right)\le P\left(d_2(\xi_1,\xi_2)>\vep^2\right)
\]
for all $\vep\in(0,1)$ and random measures $\xi_1,\xi_2$ which are almost surely in $\mathcal C_2$, \eqref{p1.eq4} implies that
\begin{equation}\label{rem.r1.eq1}
 d_4\left(\eim(\widetilde W_N/N),\xi\right)\prob0\,,
\end{equation}
as $N\to\infty$.
Thus, the assertion of Proposition \ref{p1} is stronger than that of Theorem \ref{discrete.t1} in the special case when the spectral measure of the input process is discrete.
\end{remark}

A few facts from the literature will be used in the proof of Proposition \ref{p1}, which we shall now list below. The first fact is essentially a consequence of the well known results that any two norms on a finite dimensional vector space are equivalent and that the square root is a continuous function on the set of non-negative definite (n.n.d) matrices. In the following, for a square matrix $C$, ``$C\ge0$'' means that $C$ is n.n.d., and $C^{1/2}$ denotes its n.n.d. square root.
\comment{correction-36}

\begin{fact}\label{discrete.f1}
  Suppose that for every $N\ge1$, $B_N$ is an $N\times p$ matrix, where $p$ is a fixed finite integer. Assume that
\[
 \lim_{N\to\infty}(B_N^TB_N)(i,j)=C(i,j)\text{ for all }1\le i,j\le p\,.
\]
Then $C\ge0$, and for any $p\times p$ symmetric matrix $P$
\begin{equation}\label{discrete.l1.claim}
 \lim_{N\to\infty}\Tr\left[\left\{(B_N^TB_N)^{1/2}P(B_N^TB_N)^{1/2}-C^{1/2}PC^{1/2}\right\}^2\right]=0\,.
\end{equation}
\end{fact}

The next fact is a trivial consequence of the Sylvester's determinant theorem

\begin{fact}\label{discrete.f2}
 Suppose that $B$ and $P$ are $N_1\times N_2$ and $N_2\times N_2$ matrices respectively, the latter being symmetric. Then,
\[
 \eim\left(BPB^T\right)=\eim\left((B^TB)^{1/2}P(B^TB)^{1/2}\right)\,.
\]
\end{fact}
\comment{correction-37}

The next fact that we shall use follows from Corollary 5.3 on page 115 in \cite{markus:1964}. A detailed survey of results similar to this one can be found in Chapter 13 of \cite{bhatia:2007}.

\begin{fact}\label{discrete.f3}
 For symmetric matrices $A$ and $B$ of the same size and a positive { even} integer $p$,
\[
 d_p\left(\eim(A),\eim(B)\right)\le\Tr^{1/p}\left[(A-B)^p\right]\,.
\]
\end{fact}

The following fact, the proof of which is an easy exercise, will also be needed.

\begin{fact}\label{f10}
 For all $x,y\in\bbr$, the following limits exist:
 \[
  \lim_{N\to\infty}\frac1N\sum_{k=1}^N\sin(kx)\sin(ky)\,,
 \]
\[
  \lim_{N\to\infty}\frac1N\sum_{k=1}^N\sin(kx)\cos(ky)\,,
\]
and
\[
  \lim_{N\to\infty}\frac1N\sum_{k=1}^N\cos(kx)\cos(ky)\,.
\]
Furthermore, for all $x,y\in\bbr$,
\begin{equation}\label{discrete.t2.eq5}
\lim_{N\to\infty}N^{-2}\sum_{i,j=1}^N\left[\cos(ix+jy)+\cos(iy+jx)\right]^2
\end{equation}
exists, and is strictly positive.
\end{fact}

\begin{proof}[Proof of Proposition \ref{p1}]
 Recalling the definition of $\widetilde W_{N,n}$ from \eqref{eq.deftildewnn}, in view of Fact \ref{f4}, it suffices to show that there exists a random measure $\xi_n$ which is almost surely in $\mathcal C_2$ such that
\begin{equation}\label{discrete.p1.eq2}
 d_2\left(\eim(\widetilde W_{N,n}/N),\xi_n\right)\prob0\,,
\end{equation}
as $N\to\infty$ for all fixed $n\ge1$, and
\begin{equation}\label{discrete.p1.eq3}
 \lim_{n\to\infty}\limsup_{N\to\infty}P\left[d_2\left(\eim(\widetilde W_{N,n}/N),\eim(\widetilde W_{N}/N)\right)>\vep\right]=0\text{ for all }\vep>0\,.
\end{equation}

Proceeding towards showing \eqref{discrete.p1.eq2}, fix $n\ge1$. By a relabeling, it is easy to see that for all $N\ge1$,
\[
\widetilde W_{N,n}(i,j)=\sum_{k=1}^{4n}Y_k\left[u_k(i)v_k(j)+v_k(i)u_k(j)\right]\,,
\]
where $Y_1,Y_2,\ldots,Y_{4n}$ are normal random variables which are {\bf not necessarily independent}, and for each $k$, there exists $w_k\in\bbr$ such that either
\[
 u_k(i)=\sin(iw_k)\text{ for all }i\,,
\]
or
\[
 u_k(i)=\cos(iw_k)\text{ for all }i\,,
\]
and a similar assertion holds for $v_k$ with $w_k$ replaced by some $z_k$. For $N\ge1$ and $1\le k\le8n$, define an $N\times1$ vector $u_{kN}$ by
\[
 u_{kN}(i)=u_k(i),\,1\le i\le N\,,
\]
and similarly define the vector $v_{kN}$. Next define an $N\times8n$ matrix
\[
 B_N:=\left[\sqrt{|Y_1|}u_{1N}\,\,\sqrt{|Y_1|}v_{1N}\ldots\sqrt{|Y_{4n}|}u_{4n\,N}\,\,\sqrt{|Y_{4n}|}v_{4n\,N}\right]\,,
\]
and a $8n\times8n$ symmetric matrix $P$ by
\[
 P(i,j):=\left\{\begin{array}{ll}
                 \sgn(Y_k),&\text{if }i=2k-1\text{ and }j=2k\text{ for some }k\,,\\
                 \sgn(Y_k),&\text{if }i=2k\text{ and }j=2k-1\text{ for some }k\,,\\
                 0,&\text{otherwise },
                \end{array}
\right.
\]
for all $1\le i,j\le8n$. Then, it is easy to see that
\[
\widetilde W_{N,n}=B_NPB_N^T,\,N\ge1\,.
\]
Fact \ref{discrete.f2} implies that
\begin{equation}\label{discrete.p1.eq1}
 \eim\left(\widetilde W_{N,n}/N\right)=\eim\left(\frac1N(B_N^TB_N)^{1/2}P(B_N^TB_N)^{1/2}\right),\,N\ge1\,.
\end{equation}
By Fact \ref{f10}, it follows that there exists a $8n\times8n$ matrix $C_n$ such that
\[
 \lim_{N\to\infty}\frac1N(B_N^TB_N)(i,j)=C_n(i,j)\text{ almost surely}\,,
\]
for all $1\le i,j\le8n$. Facts \ref{discrete.f1} and \ref{discrete.f3} ensure that
\[
\lim_{N\to\infty} d_2\left(\eim\left(\frac1N(B_N^TB_N)^{1/2}P(B_N^TB_N)^{1/2}\right),\eim\left(C_n^{1/2}PC_n^{1/2}\right)\right)=0\,,
\]
almost surely, which with the aid of \eqref{discrete.p1.eq1} ensures \eqref{discrete.p1.eq2}, with 
\[
\xi_n:=\eim(C_n^{1/2}PC_n^{1/2})\,. 
\]

For \eqref{discrete.p1.eq3}, it suffices to show that
\[
\lim_{n\to\infty}\limsup_{N\to\infty}\E\left[d_2^2\left(\eim(\widetilde W_{N,n}/N),\eim(\widetilde W_{N}/N)\right)\right]=0\,.
\]
To that end, notice that by Fact \ref{discrete.f3},
\begin{eqnarray*}
 \E\left[d_2^2\left(\eim(\widetilde W_{N,n}/N),\eim(\widetilde W_{N}/N)\right)\right]
 &\le&N^{-2}\E\Tr\left[(\widetilde W_{N,n}-\widetilde W_{N})^2\right]\\
 &\le&4\sum_{k=n+1}^\infty a_k\,,
\end{eqnarray*}
the last line following from arguments analogous to those leading from \eqref{ac.t1.eq1} to \eqref{ac.t1.eq2}. This establishes \eqref{discrete.p1.eq3} which along with \eqref{discrete.p1.eq2} and Fact \ref{f4} shows the existence of $\xi$ which is almost surely in $\mathcal C_2$, and satisfies
\begin{equation}\label{discrete.p1.eq5}
 d_2(\xi_n,\xi)\prob0\text{ as }n\to\infty\,,
\end{equation}
and \eqref{p1.eq4}. This completes the proof of Proposition \ref{p1}.
\end{proof}

\begin{proof}[Proof of Theorem \ref{discrete.t1}]
 By the arguments mentioned in Remark \ref{rem.r1}, \eqref{rem.r1.eq1} follows from Proposition \ref{p1}. In view of that, to complete the proof of \eqref{discrete.t1.eq4}, all that needs to be shown is 
\begin{equation}\label{discrete.t1.eq1}
 d_4\left(\eim(\widetilde W_N/N),\eim(W_N/N)\right)\prob0\text{ as }N\to\infty\,.
\end{equation}
To that end, recall \eqref{eq15}, and the definition of $W_{N,n}$ from \eqref{eq.defwnn}. By the triangle inequality, it follows that for all $N,n\ge1$,
\begin{eqnarray*}
&&d_4\left(\eim(\widetilde W_N/N),\eim(W_N/N)\right)\\
&\le&d_4\left(\eim(\widetilde W_N/N),\eim((\widetilde W_N+W_{N,n})/N)\right)\\
&&+d_4\left(\eim((\widetilde W_N+W_{N,n})/N),\eim(W_N/N)\right)\,.
\end{eqnarray*}
By Fact \ref{discrete.f3},
\comment{correction-38} 
it follows that
\begin{eqnarray}
&&\E\left[d_4^4\left(\eim(\widetilde W_N/N),\eim((\widetilde W_N+W_{N,n})/N)\right)\right]\nonumber\\
&\le&\E\left[\Tr((W_{N,n}/N)^4)\right]
\to0\text{ as }N\to\infty\,,\label{discrete.t1.eq2}
\end{eqnarray}
for all fixed $n\ge1$ using \eqref{ac.p1.eq1} with $m=2$. In order to show \eqref{discrete.t1.eq1}, it suffices to prove that
\[
 \lim_{n\to\infty}\limsup_{N\to\infty}P\left[d_4\left(\eim((\widetilde W_N+W_{N,n})/N),\eim(W_N/N)\right)>\vep\right]=0\,,
\]
for all $\vep\in(0,1)$. To that end fix such a $\vep$, and notice by the arguments in Remark \ref{rem.r1},
\begin{eqnarray*}
&&P\left[d_4\left(\eim((\widetilde W_N+W_{N,n})/N),\eim(W_N/N)\right)>\vep\right]\\
&\le&P\left[d_2\left(\eim((\widetilde W_N+W_{N,n})/N),\eim(W_N/N)\right)>\vep^2\right]\\
&\le&\vep^{-4}\E\left[d_2^2\left(\eim((\widetilde W_N+W_{N,n})/N),\eim(W_N/N)\right)\right]\\
&\le&\vep^{-4}N^{-2}\E\left[\Tr\left[(\widetilde W_N+W_{N,n}-W_N)^2\right]\right]\\
&\le&C\sum_{i,j\in\bbz:|i|\vee|j|>n}c_{i,j}^2\,,
\end{eqnarray*}
for some finite constant $C$ which is independent of $N$ and $n$. In the above calculation, the second last line follows from Fact \ref{discrete.f3}, and the last line is analogous to \eqref{ac.t1.eq2}. This shows \eqref{discrete.t1.eq1} which in turn proves \eqref{discrete.t1.eq4}.

In order to complete the proof of Theorem \ref{discrete.t1}, all that needs to be shown is that the distribution of $\xi$ is determined by $\nu_d$. That is, however, obvious from \eqref{p1.eq4} and the fact that the spectral measure of the stationary process $(Z_{i,j}:i,j\in\bbz)$ is $\nu_d$.
\end{proof}

\begin{remark}
 The only reason that in \eqref{discrete.t1.eq4}, $d_4$ cannot be changed to $d_2$ is that the limit \eqref{discrete.t1.eq2} will become false if the index $4$ is replaced by $2$. Every other step in the above proof goes through perfectly fine for $d_2$.
\end{remark}

We next proceed towards proving Theorem \ref{discrete.t2}. For that, we shall need the following two lemmas.

\begin{lemma}\label{discrete.l2}
 Suppose that $G_1,G_2,\ldots$ are i.i.d. $N(0,1)$ random variables and $\{\alpha_{jk}:j,k\in\bbz\}$ are deterministic numbers such that
\[
 \sum_{j,k=1}^{2n}\alpha_{jk}G_jG_k\prob Z\,,
\]
as $n\to\infty$, for some finite random variable $Z$. If $\alpha_{11}\neq0$, then $Z$ has a continuous distribution.
\end{lemma}

\begin{proof}
 The given hypothesis implies that
\[
 G_1\sum_{j=2}^{2n}\alpha_{1j}G_j+\sum_{j,k=2}^{2n}\alpha_{jk}G_jG_k\prob Z-\alpha_{11}G_1^2\,,
\]
as $n\to\infty$. By passing to a subsequence, we get a family of random variables $(X_n,Y_n:n\ge1)$ which is independent of $G_1$ such that
\[
 G_1X_n+Y_n\to Z-\alpha_{11}G_1^2,\text{ almost surely, as }n\to\infty\,.
\]
From here, by conditioning on $G_1$ and using the independence, it follows that  $X_n$ and $Y_n$ converge almost surely. Therefore,  there exist random variables $X$ and $Y$  such that $G_1$ is independent of $(X,Y)$ and
\[
 Z=\alpha_{11}G_1^2+G_1X+Y\text{ a.s.}\,.
\]
Since $\alpha_{11}\neq0$, for all $z\in\bbr$, it holds that
\[
 P(Z=z)=\int_{\bbr^2}P\left[\alpha_{11}G_1^2+G_1x+y=z\right]P(X\in dx,Y\in dy)=0
\]
because for every fixed $x$ and $y$, the integrand is zero. This completes the proof. 
\end{proof}

\begin{lemma}\label{discrete.l3}
 Suppose that $H_1,\ldots,H_n$ are square integrable and pairwise uncorrelated random variables, each having variance one.
\comment{correction-39} 
 Assume that $\{\alpha_{jN}:1\le j\le n,N\ge1\}$ are deterministic numbers such that the random variable
\[
 \sum_{j=1}^n\alpha_{jN}H_j\prob Z,
\]
 as $N\to\infty$. Then,
\begin{equation}\label{discrete.l3.claim1}
 \alpha_j:=\lim_{N\to\infty}\alpha_{jN}\text{ exists, for all }1\le j\le n\,,
\end{equation}
and
\[
 Z=\sum_{j=1}^n\alpha_jH_j\text{ a.s.}\,.
\]
\end{lemma}

\begin{proof}
Our first claim is that  
\begin{equation}\label{discrete.l3.eq1}
 \limsup_{N\to\infty}\sum_{j=1}^n\alpha_{jN}^2<\infty\,.
\end{equation}
For the sake of contradiction, assume that the above is false, that is, there exist integers $1\le N_1<N_2<N_3<\ldots$ such that
\[
 \lim_{k\to\infty}\sum_{j=1}^n\alpha_{jN_k}^2=\infty\,.
\]
Define
\[
 \beta_{jk}:=\alpha_{jN_k}/\sqrt{\sum_{i=1}^n\alpha_{iN_k}^2},\,1\le j\le n,k\ge1\,.
\]
Clearly,
\[
 \sum_{j=1}^n\beta_{jk}^2=1\text{ for all }k\ge1\,,
\]
and hence there exists $1\le k_1<k_2<\ldots$ such that
\[
\beta_j:= \lim_{l\to\infty}\beta_{jk_l}\text{ exists, for all }1\le j\le n\,.
\]
Therefore, 
\[
 \sum_{j=1}^n\beta_{jk_l}H_j\to\sum_{j=1}^n\beta_jH_j\text{ a.s., as }l\to\infty\,.
\]
On the other hand, notice that
\[
 \sum_{j=1}^n\beta_{jk_l}H_j=\frac{\sum_{j=1}^n\alpha_{jN_{k_l}}H_j}{\sqrt{\sum_{i=1}^n\alpha_{iN_{k_l}}^2}}\,.
\]
As $l\to\infty$, the numerator of the right hand side converges in probability to $Z$ by the given hypothesis, whereas the denominator goes to infinity. As a result, the right hand side converges to $0$ in probability. Therefore, 
\[
 \sum_{j=1}^n\beta_jH_j=0\text{ a.s.}\,.
\]
This clearly is a contradiction because in view of the assumption of $H_j$'s being pairwise uncorrelated with variance one, it follows that the left hand side has variance one. Therefore, \eqref{discrete.l3.eq1} follows.
 
Once again, for the sake of contradiction, let us assume that \eqref{discrete.l3.claim1} is false. That, in view of \eqref{discrete.l3.eq1}, means there are sequences $1\le N_1<N_2<\ldots$ and $1\le M_1<M_2<\ldots$ such that the limits
\[
 \lim_{k\to\infty}\alpha_{jN_k}=:\alpha_j^{(1)}
\]
and
\[
 \lim_{k\to\infty}\alpha_{jM_k}=:\alpha_j^{(2)}
\]
exist for $1\le j\le n$, and
\[
 \sum_{j=1}^n\left(\alpha_j^{(1)}-\alpha_j^{(2)}\right)^2>0\,.
\]
Therefore,
\[
 \sum_{j=1}^n\left(\alpha_{jN_k}-\alpha_{jM_k}\right)H_j\to\sum_{j=1}^n\left(\alpha_j^{(1)}-\alpha_j^{(2)}\right)H_j\text{ a.s., as }k\to\infty\,,
\]
and the right hand side has a positive variance. This clearly contradicts the hypothesis because it implies that the left hand side converges to zero in probability. This completes the proof of \eqref{discrete.l3.claim1}. The next claim follows from this one.
\end{proof}

\begin{proof}[Proof of Theorem \ref{discrete.t2}]
In view of \eqref{discrete.eq.bound} and \eqref{discrete.p1.eq2}, it follows that
\begin{equation}\label{discrete.t2.eq1}
 \int_\bbr x^2\xi_n(dx)\prob\int_\bbr x^2\xi(dx)\text{ as }n\to\infty\,,
\end{equation}
where $\xi_n$ is as in \eqref{discrete.p1.eq2}. The content of the proof is in showing that there exists real numbers $\{\alpha_{ijkl}:1\le i,j\le2,\,k,l\ge1\}$ such that
\begin{equation}\label{discrete.t2.eq2}
 \int_\bbr x^2\xi_n(dx)=\sum_{i,j=1}^2\sum_{k,l=1}^n\alpha_{ijkl}V_{ik}V_{jl}\text{ for all }n\ge1\,,
\end{equation}
where $V_{ik}$ is as in \eqref{eq.defz}. Furthermore, since a premise of the result is that $\nu_d$ is non-null, we assume without loss of generality that $a_1>0$ which is as in \eqref{neweq1}. Based on that, it will be shown that
\begin{equation}\label{discrete.t2.eq3}
 \alpha_{1111}>0\,.
\end{equation}
Lemma \ref{discrete.l2} along with \eqref{discrete.t2.eq1} - \eqref{discrete.t2.eq3} will establish that $\int_\bbr x^2\xi(dx)$ has a continuous distribution, and thus the claim of Theorem \ref{discrete.t2} will follow.

To that end, notice that by \eqref{discrete.eq.bound} and \eqref{discrete.p1.eq2}, it follows that
\begin{equation}\label{discrete.t2.eq4}
 N^{-2}\Tr\left[\widetilde W_{N,n}^2\right]\prob\int_\bbr x^2\xi_n(dx)\text{ as }N\to\infty\,,
\end{equation}
for all fixed $n\ge1$, where $\widetilde W_{N,n}$ is as in \eqref{eq.deftildewnn}. It is easy to see that
\[
 \widetilde W_{N,n}=\sum_{i=1}^2\sum_{k=1}^nV_{ik}A_{ikN}\,,
\]
where $A_{ikN}$ are $N\times N$ deterministic matrices defined by
\[
 A_{ikN}(u,v):=\left\{\begin{array}{ll}
                      \sqrt{a_k}\left[\cos(ux_k+vy_k)+\cos(vx_k+uy_k)\right],&i=1\,,\\
                      \sqrt{a_k}\left[\sin(ux_k+vy_k)+\sin(vx_k+uy_k)\right],&i=2\,,
                      \end{array}
\right.
\]
for all $1\le u,v\le N$. Therefore,
\[
 N^{-2}\Tr\left[\widetilde W_{N,n}^2\right]=\sum_{i,j=1}^2\sum_{k,l=1}^nV_{ik}V_{jl}N^{-2}\Tr\left(A_{ikN}A_{jlN}\right)\text{ for all }N,n\ge1\,.
\]
Since by \eqref{discrete.t2.eq4} for every fixed $n$ the right hand side converges in probability, and the random variables $\{V_{ik}V_{jl}:1\le i,j\le2,1\le k,l\le n\}$ are uncorrelated, Lemma \ref{discrete.l3} implies that
\[
 \alpha_{ijkl}:=\lim_{N\to\infty}N^{-2}\Tr\left(A_{ikN}A_{jlN}\right)\text{ exists for all }1\le i,j\le2,\,k,l\ge1\,,
\]
and that \eqref{discrete.t2.eq2} holds. Finally, notice that \eqref{discrete.t2.eq3} follows from \eqref{discrete.t2.eq5} in Fact \ref{f10}, implying that the variance of $\int_{\bbr}x^2\xi(dx)$ is positive. This completes the proof of Theorem \ref{discrete.t2}.
\end{proof}

\section{Appendix}


\subsection*{Proof of Fact \ref{f4}}
The authors chose to give a proof since they could not find a ready reference for the fact. For the same, we shall use the following intermediate fact which is very similar to Theorem 3.5, page 56 in \cite{resnick:2007}, and follows by the same arguments as there. 

\begin{fact}\label{fact.t1}
 Let $(\Sigma,d)$ be a metric space, and let $(\Omega,{\mathcal A},P)$ be a probability space. Suppose that $(X_{mn}:1\le m,n\le\infty)$ is a family of random elements in $\Sigma$, that is, measurable maps from $\Omega$ to $\Sigma$, the latter being equipped with the Borel $\sigma$-field induced by $d$. Assume that 
\begin{itemize}
\item for all fixed $1\le m<\infty$,
\[
 d(X_{mn},X_{m\infty})\prob0\,,
\]
as $n\to\infty$,
\item as $m\to\infty$,
\[
 d(X_{m\infty},X_{\infty\infty})\prob0\,,
\]
\item and, for all $\vep>0$,
\[
 \lim_{m\to\infty}\limsup_{n\to\infty}P\left[d(X_{mn},X_{\infty n})>\vep\right]=0\,.
\]
\end{itemize}
Then,
\[
 d(X_{\infty n},X_{\infty\infty})\prob0\,,
\]
as $n\to\infty$.
\end{fact}

\begin{proof}[Proof of Fact~\ref{f4}]
 In order to complete the proof using Fact \ref{fact.t1}, it suffices to show the existence of $X_{\infty\infty}$ such that
\begin{equation}\label{f4.proof.eq1}
 d(X_{m\infty},X_{\infty\infty})\prob0\text{ as }m\to\infty\,.
\end{equation}
To that end, using assumption (2), we choose integers $1\le M_1<M_2<\ldots$ such that
\begin{equation}\label{f4.proof.eq3}
 \limsup_{n\to\infty}P\left[d(X_{mn},X_{\infty n})>k^{-2}\right]\le k^{-2}\text{ for all }m\ge M_k,k\ge1\,.
\end{equation}
For fixed $k\ge1$, and $m_1,m_2\ge M_k$, it follows that
\begin{eqnarray}
\nonumber&&P\left[d(X_{m_1\infty},X_{m_{2}\infty})>4k^{-2}\right]\\
\label{f4.proof.q1}&\le&P\left[d(X_{m_1\infty},X_{m_1n})>k^{-2}\right]\\
\label{f4.proof.q2}&&+P\left[d(X_{m_1n},X_{\infty n})>k^{-2}\right]\\
\label{f4.proof.q3}&&+P\left[d(X_{m_{2}n},X_{\infty n})>k^{-2}\right]\\
\label{f4.proof.q4}&&+P\left[d(X_{m_{2}\infty},X_{m_{2}n})>k^{-2}\right]\,.
\end{eqnarray}
for all $n\ge1$. By \eqref{f4.proof.eq3}, limit superior of the quantities in \eqref{f4.proof.q2} and \eqref{f4.proof.q3}, as $n\to\infty$, are at most $k^{-2}$. By assumption (1), the quantities in \eqref{f4.proof.q1} and \eqref{f4.proof.q4} converge to zero as $n\to\infty$. Therefore, it follows that
\begin{equation}\label{f4.proof.eq2}
 P\left[d(X_{m_1\infty},X_{m_{2}\infty})>4k^{-2}\right]\le2k^{-2}\text{for all }m_1,m_2\ge M_k,\,k\ge1\,.
\end{equation}
In particular,
\[
 P\left[d(X_{M_k\infty},X_{M_{k+1}\infty})>4k^{-2}\right]\le2k^{-2},\,k\ge1\,.
\]
By the Borel Cantelli Lemma, it follows that 
\[
 P\left[d(X_{M_k\infty},X_{M_{k+1}\infty})>4k^{-2}\text{ for at most finitely many }k\text{'s}\right]=1\,.
\]
It is easy to see that on the above event, the sequence $\{X_{M_k\infty}\}$ is Cauchy, and hence convergent because $\Sigma$ is assumed to be complete. Therefore, there exists $X_{\infty\infty}$ such that
\begin{equation}\label{f4.proof.eq4}
 X_{M_k\infty}\to X_{\infty\infty}\text{ a.s., as }k\to\infty\,.
\end{equation}
To complete the proof of \eqref{f4.proof.eq1}, fix $\vep>0$. Let $k_1\ge1$ be such that $k_1^{-2}\le\vep/8$. 
Clearly, \eqref{f4.proof.eq4} also holds in probability, and therefore there exists $k_2$ such that
\begin{equation}\label{f4.proof.eq5}
 P\left[d(X_{M_k\infty}, X_{\infty\infty})>\vep/2\right]\le\vep/2\text{ for all }k\ge k_2\,.
\end{equation}
For all $m\ge M_{k_1\vee k_2}$,
\begin{eqnarray*}
&& P\left[d(X_{m\infty},X_{\infty\infty})>\vep\right]\\
&\le&P\left[d(X_{m\infty},X_{M_{k_1\vee k_2}\infty})>4k_1^{-2}\right]+P\left[d(X_{M_{k_1\vee k_2}\infty},X_{\infty\infty})>\vep/2\right]\,.
\end{eqnarray*}
The first term on the RHS is at most $2k_1^{-2}$ by \eqref{f4.proof.eq2}, which is less than $\vep/2$. The second term on the RHS is at most $\vep/2$ by \eqref{f4.proof.eq5}. This shows \eqref{f4.proof.eq1} and thus completes the proof.
\end{proof}

\subsection*{Proof of \eqref{ac.p1.eq2} in Proposition \ref{ac.p1}}\comment{correction-13(d)}
It may help the reader to recall the Definitions \ref{defn.cat} and \ref{defn.pm} because the proof is built on those. We start with the following definition.

\begin{defn}
For $N,m,n\ge1$, the set of \emph{compound pair matched tuples with parameters $N,n,2m$}, denoted by $CPM(N,n,2m)$ is the class of all $i\in\{1,\ldots,N\}^{2m}$ such that there exists $\pi\in{\mathcal P}(2m)$ satisfying
\[
\left|i_{u}\wedge i_{\gamma(u)}-i_{v}\wedge i_{\gamma(v)}\right|\vee\left|i_{u}\vee i_{\gamma(u)}-i_{v}\vee i_{\gamma(v)}\right|\le2n+1\,,
\]
for all $(u,v)\in\pi$, where
\[
\gamma:\{1,\ldots,2m\}\to\{1,\ldots,2m\}
\]
is defined by
\begin{equation}\label{eq.defgamma}
\gamma(j):=
\begin{cases}
j-1\,,&j\notin\{1,m+1\}\,,\\
m\,,&j=1\,,\\
2m\,,&j=m+1\,.
\end{cases}
\end{equation}
For all $\sigma_1,\sigma_2\in NC_2(m)$ and $k_t\in S(\sigma_t)$ for $t=1,2$, define
\[
CompCat(N,n,\sigma_1,k_1,\sigma_2,k_2):=\Bigl\{(i,j)\in Cat(N,n,\sigma_1,k_1)
\]
\[
\times Cat(N,n,\sigma_2,k_2):\min_{1\le l\le m}|i_l-j_l|>2n+1\Bigr\}\,.
\]
\end{defn}

The following lemma can be proven by standard combinatorial tools and hence the proof is skipped.

\begin{lemma}\label{appendix.l1}
If $m\ge1$ is odd, then
\begin{equation}\label{appendix.l1.eq1}
\lim_{N\to\infty}N^{-(m+2)}\#CPM(N,n,2m)=0\,.
\end{equation}
If $m\ge1$ is even, then
\begin{equation}\label{appendix.l1.eq2}
\lim_{N\to\infty}N^{-(m+2)}\#\Biggl[CPM(N,n,2m)\setminus\bigcup_{\sigma_1,\sigma_2\in NC_2(m)}\,\bigcup_{k_t\in S(\sigma_t),t=1,2}
\end{equation}
\[
Cat(N,n,\sigma_1,k_1)\times Cat(N,n,\sigma_2,k_2)\Biggr]=0\,,
\]
and for all $\sigma_1,\sigma_2\in NC_2(m)$ and $k_t\in S(\sigma_t)$ for $t=1,2$, 
\begin{eqnarray}
\nonumber\#CompCat(N,n,\sigma_1,k_1,\sigma_2,k_2)&=&\#Cat(N,n,\sigma_1,k_1)\#Cat(N,n,\sigma_2,k_2)\\
\label{appendix.l1.eq3}&&+o(N^{m+2})\,,
\end{eqnarray}
as $N\to\infty$.
\end{lemma}

\begin{proof}
[Proof of \eqref{ac.p1.eq2} in Proposition \ref{ac.p1}]
Fix $m,n\ge1$. Note that
\[
\E\left[\Tr^2\left(W_{N,n}^m\right)\right]=\sum_{i\in\{1,\ldots,N\}^{2m}}\tilde E_i\,,
\]
where
\[
\tilde E_i:=\E\left[\prod_{l=1}^{2m}W_{N,n}(i_{\gamma(l)},i_l)\right]\,,
\]
$\gamma$ being as in \eqref{eq.defgamma}. Arguments similar to those leading to \eqref{ac.p1.eq10} and \eqref{ac.p1.eq11} respectively, imply that 
\begin{eqnarray}
\label{appendix.eq1}\E\left[\Tr^2\left(W_{N,n}^m\right)\right]&=&\sum_{i\in CPM(N,n,2m)}\tilde E_i\,,\\
\label{appendix.eq2}\sup_{N\ge1,i\in\{1,\ldots,N\}^{2m}}|\tilde E_i|&<&\infty\,.
\end{eqnarray}
The above, in view of \eqref{appendix.l1.eq1}, immediately imply \eqref{ac.p1.eq2} for $m$ odd. 

Therefore, for the  rest of the proof, we assume that $m$ is even. Equations \eqref{appendix.l1.eq2}, \eqref{appendix.eq1} and \eqref{appendix.eq2} imply that 
\begin{eqnarray}
\nonumber\E\left[\Tr^2\left(W_{N,n}^m\right)\right]&=&\sum_{\sigma_1,\sigma_2\in NC_2(m)}\,\sum_{k_t\in S(\sigma_t),t=1,2}\,\sum_{j_t\in Cat(N,n,\sigma_t,k_t),t=1,2}\tilde E_{(j_1,j_2)}\\
\label{appendix.eq3}&&+o(N^{m+2})\,,
\end{eqnarray}
where $(j_1,j_2)$ denotes the concatenation of $j_1$ and $j_2$. Fix $\sigma_1,\sigma_2\in NC_2(m)$ and $k_t\in S(\sigma_t)$ for $t=1,2$. By \eqref{appendix.l1.eq3} and \eqref{appendix.eq2}, it follows that
\begin{eqnarray*}
&&\sum_{j_t\in Cat(N,n,\sigma_t,k_t),t=1,2}\tilde E_{(j_1,j_2)}\\
&=&o(N^{m+2})+\sum_{(j_1,j_2)\in CompCat(N,n,\sigma_1,k_1,\sigma_2,k_2)}\tilde E_{(j_1,j_2)}\\
&=&o(N^{m+2})+\sum_{(j_1,j_2)\in CompCat(N,n,\sigma_1,k_1,\sigma_2,k_2)}E_{j_1}E_{j_2}\,,
\end{eqnarray*}
where $E_{j_t}$ is as defined in \eqref{ac.p1.eq12}, the last line following from \eqref{ac.p1.eq13}  and the fact that for all
$
(j_1,j_2)\in CompCat(N,n,\sigma_1,k_1,\sigma_2,k_2)
$,
it holds that
\[
\min_{1\le l\le m}|j_{1,l}-j_{2,l}|>2n+1\,,
\]
where $j_t:=(j_{t,1},\ldots,j_{t,m})$, $t=1,2$.
Equations \eqref{appendix.l1.eq3} and \eqref{appendix.eq2} imply that 
\begin{eqnarray*}
&&\sum_{(j_1,j_2)\in CompCat(N,n,\sigma_1,k_1,\sigma_2,k_2)}E_{j_1}E_{j_2}\\
&=&o(N^{m+2})+\sum_{j_t\in Cat(N,n,\sigma_t,k_t),t=1,2}E_{j_1}E_{j_2}\\
&=&o(N^{m+2})+\prod_{t=1}^2\sum_{j\in Cat(N,n,\sigma_t,k_t)}E_j\\
&=&(1+o(1))N^{m+2}\prod_{t=1}^2\prod_{(u,v)\in\sigma_t}\left[\hat f_n(k_{t,u},-k_{t,v})+\hat f_n(k_{t,v},-k_{t,u})\right]\,,
\end{eqnarray*}
the last equality following by ideas similar to those leading to \eqref{ac.p1.eq14}, where $k_t:=(k_{t,1},\ldots,k_{t,m})$. In other words, the last equality uses Lemmas \ref{comb.l1} and \ref{comb.l2}. Putting everything together and using \eqref{appendix.eq3}, it follows that
\[
\lim_{N\to\infty}N^{-(m+2)}\E\left[\Tr^2\left(W_{N,n}^m\right)\right]=\beta_{n,m}^2\,.
\]
An appeal to \eqref{ac.p1.eq1} with $2m$ replaced by $m$ completes the proof of \eqref{ac.p1.eq2}.
\end{proof}
\section*{Acknowledgement}
The authors are grateful to  Manjunath Krishnapur for helpful discussions and to Octavio E. Arizmendi for pointing out a minor error. They thank two anonymous referees for  helping to improve the exposition significantly. The research of AC and RSH is supported by the INSPIRE grant of the Department of Science and Technology, Government of India. 


\def\cprime{$'$}

\end{document}